\documentclass[preprint,12pt]{elsarticle}

\usepackage{times,amssymb,amsmath,amscd,mathrsfs}
\usepackage{amsthm}
\usepackage[dvips]{epsfig}
\usepackage{graphicx}
\usepackage{color}
\usepackage{subfig}

\newtheorem{theorem}{Theorem}

\allowdisplaybreaks

\def\u{{\bm u}}

\def\el{\nonumber}
\def\cl{\nonumber \\}

\newcommand{\bm}[1]{\mbox{\boldmath{$#1$}}}

\newcommand{\Eq}[1]{(\ref{eq:#1})}

\def\u{\bm u}
\def\vt{\bm v}

\def\f{\bm f}

\def\u{{\bm u}}
\def\wt{{\bm w}}
\def\vt{{\bm v}}

\newtheorem{remark}{Remark}
\def\u{{\bm u}}
\def\vt{{\bm v}}
\def\wt{{\bm w}}
\def\st{{\bm \sigma}}
\def\tt{{\bm \tau}}

\def\f{{\bm f}}

\def\cl{\nonumber\\}
\def\el{\nonumber}

\allowdisplaybreaks[1]
\begin{document}

\begin{frontmatter}

\title{A stabilized finite element method for the two-field and three-field Stokes eigenvalue problems}

\author[aut1,aut2]{\"{O}nder T\"{u}rk}
\ead{onder.turk@yandex.com}
\author[aut3]{Daniele Boffi}
\ead{daniele.boffi@unipv.it}
\author[aut1]{Ramon Codina\corref{cor}}
\ead{ramon.codina@upc.edu}
\address[aut1]{Universitat Polit\`{e}cnica de Catalunya, Barcelona, Spain}
\address[aut2]{Gebze Technical University, Gebze/Kocaeli, Turkey}
\address[aut3]{Universit\`{a} di Pavia, Pavia, Italy}
\cortext[cor]{Corresponding author. Tel:+34934016486.}

\begin{abstract}
In this paper, the stabilized finite element approximation of the Stokes eigenvalue problems is considered for both the two-field (displacement-pressure) and the three-field (stress-displacement-pressure) formulations. The method presented is based on a subgrid scale concept, and depends on the approximation of the unresolvable scales of the continuous solution. In general, subgrid scale techniques consist in the addition of a residual based term to the basic Galerkin formulation. The application of a standard residual based stabilization method to a linear eigenvalue problem leads to a quadratic eigenvalue problem in discrete form which is physically inconvenient. As a distinguished feature of the present study, we take the space of the unresolved subscales orthogonal to the finite element space, which promises a remedy to the above mentioned complication. In essence, we put forward that only if the orthogonal projection is used, the residual is simplified and the use of term by term stabilization is allowed. Thus, we do not need to put the whole residual in the formulation, and the linear eigenproblem form is recovered properly. We prove that the method applied is convergent, and present the error estimates for the eigenvalues and the eigenfunctions. We report several numerical tests in order to illustrate that the theoretical results are validated.
\end{abstract}


\end{frontmatter}

\section{Introduction}

The finite element approximation of eigenvalue problems has been studied extensively in recent years due to the important theoretical and practical applications. The significance of the analysis  maintains its attraction, and the approximation of eigenvalue problems is still a subject of active research. In particular, there is a wide area of research on the Stokes eigenvalue problem which can be set into different frameworks, and some abstract results   can be applied to a variety of mixed or hybrid type finite element eigenvalue approximation methods (see e.g.~\cite{mercier1981}).

In this paper, the problem under consideration consists of finding eigenvalues $\lambda \in \mathbb{R}$ and eigenfunctions $u\neq {0}$  
for a certain operator $\mathscr{L}$  on a given domain $\Omega$ such that 
\begin{align}
\label{eq:gen_evp}
\mathscr{L} u = \lambda  u\quad \hbox{in}~\Omega,
\end{align}
accompanied with appropriate boundary conditions on $\partial\Omega$.

Let ${\cal X}$ be a Hilbert space for which the variational form of \Eq{gen_evp} is well defined. After normalizing $u$, this variational form reads: find a nonzero $u \in {\cal X}$ and $\lambda \in \mathbb{R}$ such that 
\begin{align}
\label{eq:gen_evpvari}
B(u,v) = \lambda (u,v)   \quad \forall   v  \in  {\cal X},
\end{align}
where $B$ is the bilinear form associated to $\mathscr{L}$ and $(\cdot,\cdot)$ stands for the inner product in $L^2(\Omega)$. 

Let ${\cal X}_h$ be a finite dimensional space of ${\cal X}$ constructed from a finite element partition of size $h$. The Galerkin discretization of  \Eq{gen_evpvari} is: find ${0} \neq u_h \in {\cal X}_h$ and $\lambda_h \in \mathbb{R}$ such that 
\begin{align}
\label{eq:gen_evpvari_h}
B(u_h,v_h) = \lambda_h (u_h,v_h) \quad \forall  v_h  \in  {\cal X}_h.
\end{align}

It is well known that when $\mathscr{L}$ is either the two-field or three-field Stokes operator, the standard Galerkin approach necessitates an interpolation for the different fields satisfying the classical inf-sup (or Babu\v ska-Brezzi) condition. { 
Researchers might want  to avoid the use of schemes satisfying this condition.} This demand has led to many recent studies devoted to develop robust and efficient stabilized techniques for approximating the Stokes eigenvalue problem \cite{armentano2014,huang2012,liu2013,xie2014,huang2015}. It is worth noting that there are also alternative approaches. For instance, a solution procedure based on a pseudostress-velocity formulation, leading to a locally conservative scheme without using additional stabilizing terms, has been proposed in~\cite{meddahi2014}. 

In the convergence analysis of the eigenvalue problems, the most common approach is to deduce the error estimates and the rate of convergence from the well known Babu\v ska-Osborn theory \cite{Babuskaosborn1991} (see also \cite{boffi2010} for a comprehensive review of finite element approximation of general eigenvalue problems). In particular, the convergence of the eigenvalues and eigenvectors for the two-field Stokes problem using mixed formulations is analyzed in many works, including \cite{mercier1981}, \cite{boffi2010} and \cite{boffi1997}.  On the other hand, despite the extensive number of papers on finite element analysis of the eigenproblem, as well as of the source problem for the two-field Stokes operator, few works have been published on the three-field case. Considering the source problem, a stabilized finite element formulation based on a subgrid concept is presented and analyzed for the stress-displacement-pressure formulation in \cite{codina13}. As another work, a Galerkin least-square based method is proposed in \cite{bonvinp1}, with stability and convergence results given for the three-field Stokes formulation arising from viscoelastic models.  

The aim of this paper is to analyze the stabilized finite element method for the Stokes eigenvalue problem in both two-field and three-field formulations. The stabilization method applied is based on a subgrid scale concept. In this method, the unresolvable scales of the continuous solution are approximately taken into account. In general, when a stabilization technique based on a projection $\tilde{P}$ of the residual is applied to \Eq{gen_evpvari}, one obtains a  statement of the form
\begin{align}
\label{eq:evp_gen_s}
& B(u_h,v_h) - \lambda_h (u_h,v_h)   \nonumber \\ 
& \qquad  +  \sum_{K}  (\tilde{P}(-\mathscr{L}^{\ast} v_h +\lambda_h v_h ), {\alpha_{K}}\tilde{P}(\mathscr{L} u_h -\lambda_h u_h))_K =0, 
\end{align}
where $\mathscr{L}^{\ast}$ is the formal adjoint operator of $\mathscr{L}$, and ${\alpha_{K}}$ is  a stabilization matrix (if $u_h$ is vector valued) of numerical parameters defined within each element domain $K$. Here and in the following, $\sum_K$ stands for the summation over all elements of the finite element partition, and $(\cdot,\cdot)_K$ for the $L^2(K)$-inner product.

It is clear from \Eq{evp_gen_s} that in general the resulting system leads to a quadratic eigenvalue problem, which, apart from being much more demanding than a linear one, could introduce eigenpairs that converge to solutions which are not solutions of the original problem \Eq{gen_evpvari}.

In this study, the unresolved subscales are assumed to be orthogonal to the finite element space, which amounts to say that $\tilde{P} = P^\bot$, the appropriate orthogonal projection. Apart from its novelty in the context of Stokes eigenvalue problems, this choice is essential to establish the structure of the eigenproblem in its original form. Only in this way the components $u_h$ and $v_h$ in the last term of \Eq{evp_gen_s} vanish, the residual is simplified, and the use of term by term stabilization is allowed. We will show that this formulation is optimally convergent for an adequate choice of the algorithmic parameters on which the method depends. This will be done by applying the classical spectral approximation theory  of \cite{Babuskaosborn1991} to the associated source problems in the spirit of the methodology developed in \cite{boffi2010}. In the convergence analysis for the two-field problem, we will make use of the stability and convergence properties of the corresponding source problem, which are adapted from \cite{codina10} and \cite{codinab1}. For the three-field eigenvalue problem, the convergence and error estimates are based on the finite element analysis of the corresponding source problem provided in \cite{codina13}. This is the first finite element approximation to the three-field Stokes eigenvalue problem to the best of our knowledge.

\section{Problem statements}
\label{sec:prsts}

\subsection{Preliminaries}
\label{subsec:pre}

Let us introduce some notation. In the following, the space of square integrable functions in a domain $\omega$ is denoted by  $L^2(\omega)$, and the space of functions having distributional derivatives of order up to an integer $m \geq 0$ belonging to $L^2(\omega)$ by $H^m(\omega)$. The space of functions in $H^1(\omega)$ vanishing on its boundary $\partial \omega$ is denoted  by $H^1_0(\omega)$.
The $L^2(\omega)$ inner product in $\omega$ for scalars, vectors and tensors, is denoted by $(\cdot,\cdot )_\omega$, and the norm in a Banach space ${\cal X}$ is denoted by $\Vert\cdot\Vert_{\cal X}$.
In what follows, the domain subscript is dropped for the case $\omega=\Omega$, $\Vert\cdot\Vert$ represents the norm on $L^2(\Omega)$, and $\Vert\cdot\Vert_m$ stands for $\Vert\cdot\Vert_{H^m(\Omega)}$ for a positive or negative $m$.  A finite element partition of the domain $\Omega$ is denoted by ${\cal P}_h$, and $K \in {\cal P}_h$ denotes an element domain.  The diameter of the finite element partition is defined as  $h = \max\{ h_K | K \in {\cal P}_h\}$, where  $h_K$ is the diameter of the element domain $K$. For simplicity, we will assume quasi-uniform meshes. When $K$ is a domain of an element in a partition, $\Vert\cdot\Vert_K$ and $\Vert\cdot\Vert_{m,K}$ denote  $\Vert\cdot\Vert_{L^2(K)}$ and $\Vert\cdot\Vert_{H^m(K)}$, respectively. Throughout the paper, the notation  $\lesssim$ is used to denote an inequality up to a constant independent of $h$ and of the coefficients of the differential equations. {All constants involved in the analysis are dimensionless.}

\subsection{ The two-field Stokes eigenproblem}
\label{subsec:f2prst}

Let $\Omega$ be bounded and polyhedral. The two-field Stokes eigenvalue problem is as follows: find $[\u, p, \lambda]$, where $\u\neq {\bf 0}$ is the displacement or velocity field, $p$ is the pressure, and $\lambda \in \mathbb{R}$, such that 
\begin{align}
\label{eq:fevp2}
\begin{cases}
  -\mu \Delta \u + \nabla p = \lambda \u \quad & \hbox{in}~\Omega,  \\
  \nabla \cdot \u =0 \quad & \hbox{in}~\Omega, \\
	\u= \bm {0} \quad & \hbox{on}~\partial\Omega, 
\end{cases}
\end{align}
where $\mu > 0$ is a physical parameter. 
The weak form of problem \Eq{fevp2} is obtained in the  functional spaces  ${\cal V}=(H^1_0(\Omega))^d$ and ${\cal Q}=L^2(\Omega)/\mathbb{R}$. Setting  ${\cal X}_{\text{I}}={\cal V} \times {\cal Q}$, this weak form can be written as: find  $[\u,p] \in {\cal X}_{\text{I}}$ and $\lambda \in \mathbb{R}$ such that 
\begin{align}
\label{eq:fevp2w}
B_{\text{I}}([\u,p],[\vt,q]) = \lambda (\u,\vt)  \quad \forall  [\vt,q] \in  {\cal X}_{\text{I}},
\end{align}
where 
\begin{align}
B_{\text{I}}([\u,p],[\vt,q]) = \mu  (\nabla \u, \nabla \vt) - (p, \nabla \cdot \vt )  + (q, \nabla \cdot \u).  
\end{align}

It is well known that the inf-sup condition holds for the continuous problem \Eq{fevp2w}, {and the corresponding solution operator is compact.} From the spectral theory (\cite{Babuskaosborn1991}) it follows that \Eq{fevp2w} has a sequence of real eigenvalues (see also \cite{armentano2014,liu2013,B-giraultr})
\begin{align}
0 <    \lambda_1 \leq  \lambda_2  \leq  \ldots \lambda_k \ldots \leq \lim_{k \to \infty}\lambda_k =  \infty , \nonumber
\end{align}
and corresponding eigenfunctions 
\begin{align}
 [\u_1, p_1], [\u_2, p_2], \ldots,  [\u_k, p_k],   \ldots  \nonumber
\end{align}
which are assumed to satisfy 
\begin{align}
 (\u_i, \u_j) = \delta_{ij}, \quad i,j=1,2,\ldots  \nonumber
\end{align}

The standard Galerkin approximation of the variational problem can be constructed on conforming finite element spaces ${\cal V}_h \subset {\cal V}$ and ${\cal Q}_h \subset {\cal Q}$. The discrete version of problem \Eq{fevp2w} is given as follows: find  $[\u_h,p_h] \in {\cal X}_{\text{I},h} = {\cal V}_h\times {\cal Q}_h$ and $\lambda_h \in \mathbb{R}$ such that 
\begin{align}
\label{eq:fevp2w_h}
B_{\text{I}}([\u_h,p_h],[\vt_h,q_h]) = \lambda_h (\u_h,\vt_h)   \quad \forall  [\vt_h,q_h] \in  {\cal X}_{\text{I},h}.
\end{align}
{The restriction in the possible choices for the displacement and pressure spaces dictated by the inf-sup condition motivates the use of a stabilization technique to solve this problem.} The stabilized finite element formulation adopted in this paper has its roots in the variational multiscale formulation, where the continuous space ${\cal X}_{\text{I}}$ of the problem is approximated by  ${\cal X}_{\text{I},h}  \oplus \tilde{{\cal X}_{\text{I}}} $, $\tilde{{\cal X}_{\text{I}}}$ being an approximation to the complement of ${\cal X}_{\text{I},h}$ in ${\cal X}_{\text{I}}$.  

In our study, we select $\tilde{{\cal X}_{\text{I}}}$ to be approximately orthogonal to ${\cal X}_{\text{I},h}$ leading to the so-called method of orthogonal subscales \cite{codina10,codinab1,codina7}. The resulting simplified stabilized method for problem \Eq{fevp2w_h} that we shall use reads: find  $[\u_h,p_h] \in {\cal X}_{\text{I},h}$ and $\lambda_h \in \mathbb{R}$ such that 
\begin{align}
\label{eq:fevp2w_hs}
B_{\text{IS}}([\u_h,p_h],[\vt_h,q_h]) = \lambda_h (\u_h,\vt_h)   \quad \forall  [\vt_h,q_h] \in  {\cal X}_{\text{I},h},
\end{align}
where $B_{\text{IS}}([\u_h,p_h],[\vt_h,q_h])$ is defined as
\begin{align}
\label{eq:BIS}
B_{\text{IS}}([\u_h,p_h],[\vt_h,q_h]) & =B_{I}([\u_h,p_h],[\vt_h,q_h]) \nonumber \\ &+ \sum_{K} \alpha_{1K} ( P^\bot(\nabla p_h), P^\bot(\nabla q_h))_K \nonumber \\ &+  \alpha_2 ( P^\bot(\nabla \cdot \vt_h), P^\bot(\nabla \cdot \u_h)).
\end{align}
$\alpha_{1K}$ and $\alpha_{2}$ are  the stabilization parameters, which are computed as $$\alpha_{1K}=\frac{h_K^2}{ \mu }c_1, \quad \alpha_{2} = c_2 \mu $$ where $c_1$ and $c_2$ are numerical constants (see \cite{codina7} for more details on the method and the stabilization parameters). 
In the implementation of the method, a term of the form $( P^\bot (f_h), P^\bot(g_h))$ is computed as $( f_h , g_h-P(g_h))$ where the projection onto the appropriate finite element space $P(g_h)$ can either be treated implicitly or in an iterative way. 

\begin{remark}
Let us remark that in the design of the stabilization parameters one could take into account the eigenvalue, considering that $\lambda\u$ is a reactive-like term (see \cite{codina4}). The effect of neglecting it is that the estimates to be obtained will not be uniform in terms of the magnitude of the eigenvalue, but this is the same situation encountered in any Galerkin approximation of an eigenproblem, including the inf-sup stable approximation of the Stokes problem.
\end{remark}

\subsection{The three-field Stokes eigenproblem}
\label{subsec:f3prst}

The three-field Stokes eigenvalue problem is written as follows: find $[\u, p, \st]$, and  $\lambda \in \mathbb{R}$  such that 
\begin{align}
\label{eq:fevp3}
\begin{cases}
  -\nabla \cdot \st + \nabla p = \lambda \u \quad & \hbox{in}~\Omega,  \\
  \nabla \cdot \u =0 \quad & \hbox{in}~\Omega, \\
	\frac{1}{2\mu} \st- \nabla^S \u=\bm {0} \quad & \hbox{in}~\Omega,  \\
	\u= \bm {0} \quad & \hbox{on}~\partial\Omega, 
\end{cases}
\end{align}
where $\u\neq {\bf 0}$ is the displacement field, $p$ is the pressure, $\st$ is the deviatoric component of the stress field and $\nabla^S\u$ is the symmetrical part of $\nabla\u$. To write the weak form of problem \Eq{fevp3}, in addition to the 
functional spaces  ${\cal V}=(H^1_0(\Omega))^d$ and ${\cal Q}=L^2(\Omega)/\mathbb{R}$, we define   ${\cal T} = (L^2(\Omega) )^d_{\text{sym}}$ as the space of symmetric tensors of second order with square-integrable components. If we now let ${\cal X}_{\text{II}}={\cal V} \times {\cal Q} \times  {\cal T} $, the weak form of the problem can be stated as the follows: find  $[\u,p,\st] \in {\cal X}_{\text{II}}$ and $\lambda \in \mathbb{R}$ such that 
\begin{align}
\label{eq:fevp3w_v}
B_{\text{II}}([\u,p,\st],[\vt,q,\tt]) = \lambda (\u,\vt)   \quad \forall  [\vt,q, \tt] \in  {\cal X}_{\text{II}},
\end{align}
where 
\begin{align}
B_{\text{II}}([\u,p,\st],[\vt,q,\tt]) &=  (\nabla^S \vt, \st) - (p, \nabla \cdot \vt )  + (q, \nabla \cdot \u )  \nonumber\\ &+ \frac{1}{2\mu} (\st,\tt) - (\nabla^S \u, \tt). 
\end{align}

The Galerkin finite element approximation is  obtained in the usual way, by building the conforming finite element spaces ${\cal V}_h \subset {\cal V}$,  ${\cal Q}_h \subset {\cal Q}$ and ${\cal T}_h \subset {\cal T}$.  If we let ${\cal X}_{\text{II},h}={\cal V}_h \times {\cal Q}_h  \times {\cal T}_h $, the problem is now: find $[\u_h,p_h,\st_h] \in {\cal X}_{\text{II},h}$ and $\lambda_h \in \mathbb{R}$ such that 
\begin{align}
\label{eq:fevp3w_h}
B_{\text{II}}([\u_h,p_h,\st_h],[\vt_h,q_h,\tt_h]) = \lambda (\u_h,\vt_h)   \quad \forall  [\vt_h,q_h,\tt_h] \in  {\cal X}_{\text{II},h}.
\end{align}

It is obviously seen that the bilinear form $B_{\text{II}}([\u_h,p_h,\st_h],[\vt_h,q_h,\tt_h])$ is not coercive and the inf-sup condition is not satisfied unless some stringent requirements are posed on the choice of the finite element spaces. Thus, like for the two-field formulation, the purpose of the stabilization used is to avoid the use of the inf-sup conditions and, in particular, to allow equal interpolations for all the unknowns.  The same strategy as before is followed to obtain the stabilized finite element formulation, and the method becomes: find  $[\u_h, p_h, \st_h] \in {\cal X}_{\text{II},h}$ and $\lambda_h \in \mathbb{R}$ such that 
\begin{align}
\label{eq:fevp3w_hs}
B_{\text{IIS}}([\u_h,p_h,\st_h],[\vt_h,q_h,\tt_h]) = \lambda_h (\u_h,\vt_h)   \quad \forall  [\vt_h,q_h, \tt_h] \in  {\cal X}_{\text{II},h},
\end{align}
where $B_{\text{IIS}}([\u_h,p_h,\st_h],[\vt_h,q_h,\tt_h])$ is given as
\begin{align}
& B_{\text{IIS}}([\u_h,p_h,\st_h],[\vt_h,q_h,\tt_h]) = B_{\text{II}}([\u_h,p_h,\st_h],[\vt_h,q_h,\tt_h])  \nonumber\\
& \qquad + \alpha_3 ( P^\bot(\nabla^S \vt_h), P^\bot(\nabla^S \u_h))  
+ \alpha_4 ( P^\bot(\nabla \cdot \vt_h), P^\bot(\nabla \cdot \u_h))   \nonumber\\
& \qquad + \sum_{K} \alpha_{5K} ( P^\bot(\nabla q_h-\nabla \cdot \tt_h), P^\bot(\nabla p_h-\nabla \cdot \st_h))_K .  \label{eq:BIIS}
\end{align}

The stabilization parameters of this formulation are given by $$\alpha_3=2\mu c_3, \quad \alpha_4=2\mu c_4, \quad  \alpha_{5K}= \frac{h_K^2}{\mu } c_5 $$ where $c_3$, $c_4$ and  $c_5$ are numerical constants which can be taken in a wide range, as the analysis in \cite{codina13} put forth. In this paper we consider that the finite element spaces are built using equal continuous interpolation, although the extension to more general spaces, and in particular of discontinuous stresses and pressures, can be done as analyzed  in \cite{codina13}. 

\section{Numerical analysis of the source problems}
\label{sec:analysis}

As we have mentioned earlier, we aim to prove that the eigensolutions of the stabilized two-field and three-field Stokes problems converge to the solutions of the corresponding spectral problems by applying the classical spectral approximation theory presented in \cite{Babuskaosborn1991} to the associated source problems. To achieve this, we will present in this section the source problems for the two-field and three-field cases, and the essential stability and convergence results. 
At this point, let us introduce the notation for the interpolation estimates that will allow us to define the error functions of the methods.  
For any  $v \in H^{k'_v+1}(\Omega)$, $k_v$ being the degree of an approximating finite element space ${\cal W}_h $, the interpolation errors $\varepsilon_i(v)$, $i=0,1$, are derived from the interpolation estimates as
\begin{align}
\label{eq:interestgen}
\varepsilon_i(v)= h^{k''_v+1-i}\sum_{K} \Vert v\Vert_{k''_v+1,K},
\end{align} with
\begin{align}
   \inf_{v_h \in  {\cal W}_h}  \sum_{K} \Vert v-v_h\Vert_{i,K}  \lesssim \varepsilon_i(v),
\end{align}where $k''_v=\min(k_v,k'_v)$. Here we use $v$ to represent the unknown $\u$, $\st$ or $p$, and $k_v$ denotes the corresponding order of interpolation for each $v$.
In the results given below, we will define the error functions  based on these definitions.

\subsection{The two-field source problem}
\label{subsubsec:stabconv2}

The source Stokes problem  for the two-field case can be written as: given $\f \in (L^2(\Omega))^d$, find $[\u,p] \in {\cal X}_{\text{I}}$ such that
\begin{align}
\label{eq:sourcev2}
B_{\text{I}}([\u,p],[\vt,q]) = (\f,\vt)   \quad \forall  [\vt,q] \in  {\cal X}_{\text{I}}.
\end{align}
The corresponding stabilized formulation can be written as: find $[\u_h,p_h] \in {\cal X}_{\text{I,h}}$ such that
\begin{align}
\label{eq:sourcestab2}
B_{\text{IS}}([\u_h,p_h],[\vt_h,q_h]) = (\f,\vt_h)  \quad \forall [\vt_h,q_h] \in  {\cal X}_{\text{I},h},
\end{align}
where $B_{\text{IS}}$ is defined in \Eq{BIS}. 

The stability and convergence properties of the method used in \Eq{sourcestab2} are analyzed in \cite{codina10} and \cite{codinab1}, and the following theorem is a collection of immediate consequences of the results obtained therein:

\begin{theorem}
\label{th:thc_2fs}
The solution  of \Eq{sourcestab2} satisfies the stability condition
\begin{align}
\sqrt{\mu} \Vert \u_h \Vert_{1}  + \frac{1}{\sqrt{\mu}}\Vert p_h  \Vert  \lesssim \frac{1}{\sqrt{\mu}}  \Vert \f \Vert_{-1}.
\end{align}
Moreover, if the solution of the continuous problem has enough regularity, then the solution of \Eq{sourcestab2}  has the following optimal order of convergence
\begin{align}
\sqrt{\mu} \Vert \u - \u_h \Vert_{1} + \frac{1}{\sqrt{\mu}}\Vert p- p_h  \Vert  \lesssim \varepsilon_{\text{\rm I}}(h), 
\end{align} where $\varepsilon_{\text{\rm I}}(h)$ is the interpolation error given by
\begin{align}
\label{eq:Eh2f}
\varepsilon_{\text{\rm I}}(h)=\sqrt{\mu} \varepsilon_1(\u)+ \frac{1}{\sqrt{\mu}}\varepsilon_0(p). 
\end{align} 
\boxed{}
\end{theorem}
It is important to note that as the definition of $\varepsilon_{\text{I}}(h)$ suggests, the interpolation error is of order $k''$ in terms  of $h$, where  $k''=\min(k''_u,k''_p+1)$. 

Now, for reasons that will become obvious in the convergence theory given in Section \ref{subsec:stabconv2e}, we state and prove the following theorem, which asserts the $L^2$-error estimate for the displacement:

\begin{theorem}
\label{th:th2}
If the continuous problem \Eq{sourcev2} satisfies the regularity condition
\begin{align}
\label{eq:assumpthm1}
 \sqrt{\mu} \Vert \u \Vert_{2} + \frac{1}{\sqrt{\mu}}\Vert p  \Vert_{1} \lesssim \frac{1}{\sqrt{\mu}}  \Vert \f \Vert,
\end{align}
then
\begin{align}
 \sqrt{\mu} \Vert \u-\u_h \Vert  \lesssim h^{2} \left( \sqrt{\mu} \Vert \u \Vert_{2}   + \frac{1}{\sqrt{\mu}}\Vert p  \Vert_{1}  \right).
\end{align}
\end{theorem}
\begin{proof} The proof is carried out by using a duality argument. To do this, we let $[\wt,\pi] \in {\cal X}_{\text{I}}$ and consider the following adjoint problem: 
\begin{align}
\label{eq:fevp2adj}
\begin{cases}
  -\mu \Delta \wt - \nabla \pi = \frac{\mu}{\ell^2} (\u-\u_h) \quad & \hbox{in}~\Omega,  \\
  -\nabla \cdot \wt =0 \quad & \hbox{in}~\Omega, \\
	\wt= \bm {0} \quad & \hbox{on}~\partial\Omega, 
\end{cases}
\end{align}
where $\ell$ is a characteristic length introduced to maintain the dimensional consistency of the problem. The next step is to test the first and second equations in \Eq{fevp2adj} respectively with $\u-\u_h$ and $p-p_h $. Then we have
\begin{align}
\label{eq:eql3}
\frac{\mu}{\ell^2} \Vert \u - \u_h \Vert^2 &= \mu  (\nabla \wt, \nabla (\u-\u_h)) + (\pi, \nabla \cdot (\u-\u_h))  - (p-p_h, \nabla \cdot \wt ) \cl
&= B_{\text{I}}([\u-\u_h,p-p_h], [\wt,\pi])  \nonumber\\
&= B_{\text{IS}}([\u-\u_h,p-p_h], [\wt,\pi]) \nonumber\\
& - \sum_{K} \alpha_{1K} ( P^\bot(\nabla \pi), P^\bot(\nabla (p-p_h)))_K \nonumber\\
& - \alpha_2 \sum_{K}( P^\bot(\nabla \cdot \wt), P^\bot(\nabla \cdot (\u-\u_h)))_K,
\end{align}
where we have used the definition of $B_{\text{IS}}$. The last term in the expression above vanishes because $\nabla\cdot\wt = 0$, and thus only the first two terms on the right-hand-side of the last equality must be bounded. Now if we let $[\tilde{\wt}_h,\tilde{\pi}_h]$ be the best approximation to $[{\wt},{\pi}]$ in ${\cal X}_{\text{I},h}$, the first of these terms can be bounded using the consistency error coming from $B_{\text{IS}}([\u-\u_h,p-p_h],[\tilde{\wt}_h,\tilde{\pi}_h])$ as follows. Since $[\u,p]$ is the solution of the continuous problem \Eq{sourcev2}, we have
\begin{align}
\label{eq:biscontf}
B_{\text{IS}}([\u,p],[{\tilde\wt}_h,{\tilde\pi}_h]) = (\f,{\tilde\wt}_h) +\sum_{K} \alpha_{1K} ( P^\bot(\nabla p), P^\bot(\nabla {\tilde\pi}_h))_K  \nonumber \\ +\alpha_2 ( P^\bot(\nabla \cdot \u), P^\bot(\nabla \cdot {\tilde\wt}_h)).
\end{align}
Making use of  \Eq{sourcestab2} in \Eq{biscontf} and noting that $\nabla\cdot\u = 0$ yields
\begin{align}
B_{\text{IS}}([\u-\u_h,p-p_h],[{\tilde\wt}_h,{\tilde\pi}_h])  
 \lesssim \frac{h^2}{\mu} \Vert \pi\Vert_1 \Vert p\Vert_1,\el
\end{align}
where we have made use of the $H^1$-stability of the best interpolation and the expression of $\alpha_{1K}$. This same bound clearly applies to the second term in the right-hand-side of \Eq{eql3}. Regarding the stabilizing terms applied to $([\u-\u_h,p-p_h],[\wt-{\tilde\wt}_h,\pi - {\tilde\pi}_h])$, we have that
\begin{align}
& \sum_K \alpha_{1K} (P^\bot(p-p_h),P^\bot(\pi-{\tilde\pi}_h))_K\cl
& \qquad + \alpha_2 (P^\bot(\nabla\cdot\u-\nabla\cdot\u_h),P^\bot(\nabla\cdot\wt-\nabla\cdot{\tilde\wt}_h))\cl
& \qquad \lesssim \frac{h^2}{\mu} \Vert \pi\Vert_1 \Vert p\Vert_1
+ \mu \Vert \u - \u_h\Vert_1 \Vert \wt - \tilde{\wt}_h\Vert_1 \cl
& \qquad \lesssim \frac{h^2}{\mu} \Vert \pi\Vert_1 \Vert p\Vert_1 + \mu h^2 \Vert\u\Vert_2\Vert \wt\Vert_2.\el
\end{align}
On the other hand, it is easily seen that
\begin{align}
& B_{\text{I}}([\u-\u_h,p-p_h],[\wt-{\tilde\wt}_h,\pi - {\tilde\pi}_h])  \cl
& \qquad \lesssim \mu \Vert \nabla \u - \nabla \u_h \Vert h \Vert \wt\Vert_2 + \Vert p - p_h\Vert h \Vert \wt\Vert_2
 + \Vert \nabla\cdot\u - \nabla\cdot\u_h\Vert h \Vert \pi \Vert_1 \cl
& \qquad \lesssim h^2 \mu \Vert \u\Vert_2\Vert \wt\Vert_2 + h^2\Vert p\Vert_1\Vert \wt\Vert_2
+ h^2 \Vert \u\Vert_2 \Vert \pi\Vert_1.\el
\end{align}
Collecting the bounds just obtained and using them in \Eq{eql3} yields
\begin{align}
& \frac{\mu}{\ell^2} \Vert \u - \u_h \Vert^2   \lesssim
h^2 \mu \Vert \u\Vert_2\Vert \wt\Vert_2 + h^2\Vert p\Vert_1\Vert \wt\Vert_2
+ h^2\Vert \u\Vert_2 \Vert \pi\Vert_1
+ \frac{h^2}{\mu} \Vert \pi\Vert_1 \Vert p\Vert_1.\el
\end{align}
From the elliptic regularity assumption we have that
\begin{align*}
\Vert \wt \Vert_2 \lesssim \frac{1}{\ell^2} \Vert \u-\u_h \Vert, \quad 
\Vert \pi \Vert_1 \lesssim   \frac{\mu}{\ell^2} \Vert \u-\u_h \Vert,
\end{align*}
which when used in the previous bound yields the theorem.
\end{proof}
 
\subsection{The three-field source problem}
\label{subsubsec:stabconv3}

The three-field Stokes source problem can be written as: given $\f\in (L^2(\Omega))^d$,  seek $[\u,p,\st] \in {\cal X}_{\text{II}}$  such that  
\begin{align}
\label{eq:sourcev3}
B_{\text{II}}([\u,p,\st],[\vt,q,\tt]) = (\f,\vt)  \quad \forall  [\vt,q,\tt]\in  {\cal X}_{\text{II}}, 
\end{align}
and the stabilized formulation is: find $[\u_h,p_h,\st_h] \in {\cal X}_{\text{II,h}}$ such that
\begin{align}
\label{eq:sourcestab3}
B_{\text{IIS}}([\u_h,p_h,\st_h],[\vt_h,q_h,\tt_h]) = (\f,\vt_h) \quad \forall  [\vt_h,q_h,\tt_h] \in  {\cal X}_{\text{II},h},
\end{align}
where $B_{\text{IIS}}$ is defined in \Eq{BIIS}. 

The following theorem, which is proved in \cite{codina13}, asserts the stability and convergence of the finite element solution:
\begin{theorem}
\label{th:thc_45}
The   solution  of \Eq{sourcestab3} can be  bounded as
\begin{align}
\sqrt{\mu} \Vert \u_h \Vert_{1} +\frac{1}{\sqrt{\mu}}\Vert \st_h  \Vert + \frac{1}{\sqrt{\mu}}\Vert p_h  \Vert  \lesssim \frac{1}{\sqrt\mu}  \Vert \f \Vert_{-1}.
\end{align}
Moreover, if the solution of the continuous problem has enough regularity,
\begin{align}
\sqrt{\mu} \Vert \u - \u_h \Vert_{1} +\frac{1}{\sqrt{\mu}}\Vert \st-\st_h  \Vert + \frac{1}{\sqrt{\mu}}\Vert p- p_h  \Vert  \lesssim \varepsilon_{\text{\rm II}}(h),
\end{align}
where  $\varepsilon_{\text{\rm II}}(h)$ is the interpolation error given by
\begin{align}
\label{eq:Eh}
\varepsilon_{\text{\rm II}}(h)=\sqrt{\mu} \varepsilon_1(\u)+\frac{1}{\sqrt{\mu}}\varepsilon_0(\st)+ \frac{1}{\sqrt{\mu}}\varepsilon_0(p).
\end{align}\boxed{}
\end{theorem}

In the same way as in the previous section,  the definition of $\varepsilon_{\text{II}}(h)$ states that the interpolation error is of order $k''$ in terms of $h$, where $k''=\min(k''_u,k''_{\sigma}+1,k''_p+1)$ for this case. 
 
To complete the convergence analysis for the three-field source problem, below we include the theorem  stated and proved in~\cite{codina13}, which provides an  $L^2$-error estimate for the displacement:

\begin{theorem}
\label{th:th3fL2}
If the continuous three-field source problem satisfies the regularity condition
\begin{align}
\label{eq:assumpthm2}
 \sqrt{\mu} \Vert \u \Vert_{2} +\frac{1}{\sqrt{\mu}}\Vert \st  \Vert_{1}+ \frac{1}{\sqrt{\mu}}\Vert p  \Vert_{1} \lesssim \frac{1}{\sqrt{\mu}}  \Vert \f \Vert,
\end{align}
then
\begin{align}
 \sqrt{\mu} \Vert \u-\u_h \Vert  \lesssim h^{2} \left( \sqrt{\mu} \Vert \u \Vert_{2} +\frac{1}{\sqrt{\mu}}\Vert \st \Vert_{1}+ \frac{1}{\sqrt{\mu}}\Vert p  \Vert_{1}  \right).
\end{align}
\boxed{}
\end{theorem}

\section{Numerical analysis of the eigenvalue problems}
\label{sec:evpanalysis}

In this section, we aim to apply the convergence analysis to the two-field and three-field Stokes eigenproblems. Mainly, we account for the theory developed in \cite{boffi2010}, which has its roots in the abstract spectral approximation theory of Babu\v ska-Osborn. The convergence results, and the error estimates for the displacement in suitable norms obtained in the previous section, are considered as the constitutional steps to accomplish our tasks. We will report the sufficient and necessary conditions for the convergence of eigenvalues and eigenfunctions to the continuous problems, and the approximation rates for each case. 

\subsection{The two-field eigenproblem}
\label{subsec:stabconv2e}

The object of this subsection is to provide, for the two-field case, the necessary and sufficient conditions for
proving that the eigenvalues and eigenfunctions of \Eq{fevp2w_hs} converge to those of \Eq{fevp2} with no spurious solutions, and to find an estimate for the order of convergence. As already discussed, the convergence results stated in Section \ref{subsubsec:stabconv2} will be used following the spectral approximation theory with an analogous notation to that of \cite{boffi2010}. However, before proceeding, we want to emphasize that the Galerkin formulation of the two-field eigenproblem \Eq{fevp2} can be set into the framework of a standard mixed eigenvalue problem of the first type according to the classification in \cite{boffi2010} and \cite{boffi1997} as follows: find a nontrivial $\u \in {\cal V}$ and $\lambda \in \mathbb{R}$ such that for some  $p \in {\cal Q}$  
\begin{align}
\label{eq:fevp2boffi}
\begin{cases}
a_{\text{I}}(\u,\vt) + b_{\text{I}}(\vt,p) =\lambda (\u,\vt) \quad \forall  \vt \in  {\cal V}, \\ 
b_{\text{I}}(\u,q) = 0 \quad \forall  q \in {\cal Q},
\end{cases}
\end{align} 
where the  bilinear forms introduced are given by
\begin{align*}
 &a_{\text{I}}(\u,\vt)=\mu  (\nabla \u, \nabla \vt), \\
 &b_{\text{I}}(\vt,q) =(q, \nabla \cdot \vt ). 
\end{align*}
If  ${\cal V}_h \subset {\cal V}$ and  ${\cal Q}_h \subset {\cal Q}$ are the finite element spaces to approximate the solution, the Galerkin finite element approximation can be written as: find a nontrivial $\u_h \in {\cal V}_h$ and $\lambda \in \mathbb{R}$ such that for some  $p_h \in {\cal Q}_h$ there holds
\begin{align}
\label{eq:fevp2boffi_h}
\begin{cases}
a_{\text{I}}(\u_h,\vt_h) + b_{\text{I}}(\vt_h,p_h) =\lambda (\u_h,\vt_h) \quad \forall  \vt_h \in  {\cal V}_h, \\ 
b_{\text{I}}(\u_h,q_h) = 0 \quad \forall  q_h \in {\cal Q}_h.
\end{cases}
\end{align}
The convergence of the eigensolutions to \Eq{fevp2boffi_h} towards those of \Eq{fevp2boffi} is analyzed in~\cite{boffi2010}.  

Let us come back to our main task of analyzing the two-field eigenvalue problem. The existence and uniqueness of the solutions to \Eq{sourcev2}  and \Eq{sourcestab2} allows us to define the operators  $T, T_h:{\cal X} \rightarrow {\cal X}$  such that for any $\f \in {\cal X}$,  $T \f = \u$ and $T_h \f = \u_h$  are the displacement components of the solutions to  \Eq{sourcev2} and \Eq{sourcestab2}, respectively, where ${\cal X}$ can be either $(H^1_0(\Omega))^d$ or $(L^2(\Omega))^d$. 
Now, by means of Theorem~\ref{th:thc_2fs}, we can state the convergence of the discrete operator $T_h$ to the continuous operator $T$, that is to say,
\begin{align}
\label{eq:normT}
\Vert T-T_h \Vert_{{\cal L}({\cal X})}     \rightarrow 0 \quad \hbox{as }~ h \rightarrow 0,
\end{align} 
which is equivalent to convergence of the eigenvalues and eigenfunctions according to the theory given in \cite{boffi2010} and \cite{boffi2000a}. In \Eq{normT}, ${{\cal L}({\cal X})}$ denotes the space of endomorphisms in $\cal X$ and $\Vert \cdot \Vert_{{\cal L}({\cal X})}$ its natural norm.



We next present in the following theorem the error estimates for eigenvalues of the approximate problem:

\begin{theorem}
\label{th:conv2_ev}
Assume that the continuous problem satisfies the  regularity condition
\begin{align}
\label{eq:assumpt_2genk}
 \sqrt{\mu} \Vert \u \Vert_{k''+1} + \frac{1}{\sqrt{\mu}}\Vert p  \Vert_{k''} \lesssim \frac{1}{\sqrt{\mu}} h^{k''}   \Vert \f \Vert_{k''-1},
\end{align}
for $k''>0$. Then the following optimal double order of convergence holds 
\begin{align}
\left| \lambda-\lambda_h \right| \lesssim  \frac{\mu}{\ell^2}\left(\frac{h}{\ell}\right)^{2k''},
\end{align}
where $\ell$ is, as before, a characteristic length scale of the problem. 
\end{theorem}
 
\begin{proof} For ${\cal X}=(H^1_0(\Omega))^d$, from Theorem \ref{th:thc_2fs} it follows that
\begin{align*}
\sqrt{\mu} \Vert T \f-T_h \f \Vert_{1}   
&=\sqrt{\mu} \Vert \u - \u_h \Vert_{1} \\
&\lesssim \varepsilon_{\text{I}}(h) \\
&\lesssim \sqrt{\mu} \, h^{k''}  \Vert \u\Vert_{k''+1}  +\frac{1}{\sqrt{\mu}}h^{k''} \Vert p\Vert_{k''} \\
&\lesssim h^{k''} \frac{1}{\sqrt{\mu}} \Vert \f\Vert_{k''-1} \\
&\lesssim \frac{\ell^{2}}{\sqrt{\mu}}   \left(\frac{h}{\ell}\right)^{k''}\Vert \f\Vert_{1} 
\end{align*} 
where a norm embedding has been used in the last step. The proof follows from Corollary 9.8 of \cite{boffi2010}, using the definitions of $\varepsilon_{\text{I}}(h)$ and $k''$. 
\end{proof}

\begin{remark}
\label{rem2}
The assumption \Eq{assumpt_2genk} for $0<k''<1$ suffices to prove the convergence \Eq{normT} in the proof of Theorem \ref{th:conv2_ev}. On the other hand, the stronger assumption $k''=1$ is needed to obtain the $L^2$-error estimate given by Theorem \ref{th:th2}, which is essential to obtain the convergence result  \Eq{normT} if  ${\cal X}=(L^2(\Omega))^d$ is chosen.
\end{remark}

Next, we make use of Corollary 9.4 of \cite{boffi2010} to conclude that if $\lambda$ is an eigenvalue of \Eq{fevp2w_h} with algebraic multiplicity $m$, and $E=E(\lambda^{-1}){\cal X}$ is its generalized eigenspace, where $E(\lambda)$ is the Riesz spectral projection associated with $\lambda$, and if $E_h=E_h(\lambda^{-1}){\cal X}$. Then 
\begin{align*}
\hat{\delta}(E,E_h) \lesssim  \sup_{\substack {{\scriptsize\u}\in E \\ \Vert {\scriptsize\u}  \Vert_{{\cal X}}=1 }}\inf_{{\scriptsize\u}_h\in E_h} \Vert \u -\u_h  \Vert_{{\cal X}}.
\end{align*}
Having arrived at these results, one can now prove the following:

\begin{theorem}
\label{th:conv2_ef} 
Let $\u$ be a unit eigenfunction solution of \Eq{fevp2w} associated to the eigenvalue $\lambda$ of multiplicity $m$, and let ${\bm\phi}^1_h, \ldots ,{\bm \phi}^m_h$ be the eigenfunctions associated with the $m$ discrete eigenvalues solution of \Eq{fevp2w_hs} converging to $\lambda$. Then there exists a discrete eigenfunction $\u_h \in \hbox{\rm span}\{{\bm\phi}^1_h, \ldots ,{\bm\phi}^m_h\}$ such that 
\begin{align}
\Vert \u - \u_h \Vert_{1} \lesssim  h^{k''}  \Vert \u  \Vert_{k''+1}. 
\end{align}
\end{theorem}

\subsection{The three-field eigenproblem}
\label{subsec:stabconv3e}

We first remark that it is possible to obtain a standard mixed formulation for the  three-field eigenproblem \Eq{fevp3} as follows (see also  \cite{fortin-pierre-1989-1}):  find  $\u \in {\cal V}$ and $\lambda \in \mathbb{R}$ such that for some  $[\st,p] \in {\cal T} \times {\cal Q}$ there holds
\begin{align}
\label{eq:fevp3boffi}
\begin{cases}
a_{\text{II}}([\st,p],[\tt,q]) + b_{\text{II}}([\tt,q],\u) =0  \quad \forall  [\tt,q] \in  {\cal T} \times {\cal Q},\\
b_{\text{II}}([\st,p],\vt) =-\lambda (\u,\vt)   \quad \forall  \vt \in  {\cal V} ,
\end{cases}
\end{align} where we have introduced the following bilinear forms:
\begin{align*}
 &a_{\text{II}}([\st,p],[\tt,q]) =\frac{1}{2\mu} (\st,\tt),   \\
 &b_{\text{II}}([\st,p],\vt) =- (\tt-q{\bm I}, \nabla^S \u). 
\end{align*}

As before, once the finite element spaces ${\cal V}_h \subset {\cal V}$,  ${\cal Q}_h \subset {\cal Q}$ and ${\cal T}_h \subset {\cal T}$ have been constructed, the discrete eigenvalue problem can be written as follows: find  $\u_h \in {\cal V}_h$ and $\lambda \in \mathbb{R}$ such that for some  $[\st_h,p_h] \in {\cal T}_h \times {\cal Q}_h$ there holds
\begin{align}
\label{eq:fevp3boffi_h}
\begin{cases}
   a_{\text{II}}([\st_h,p_h],[\tt_h,q_h]) + b_{\text{II}}([\tt_h,q_h],\u_h) =0  \quad \forall  [\tt_h,q_h] \in  {\cal T}_h \times {\cal Q}_h,\\
   b_{\text{II}}([\st_h,p_h],\vt_h) =-\lambda_h (\u_h,\vt_h) \quad \forall  \vt_h \in  {\cal V}_h .
\end{cases}
\end{align}
This is a standard mixed eigenvalue problem of the second type according to the classification in \cite{boffi2010}, and can be analyzed by using the abstract theory given there. 

Our ultimate purpose in this section is to provide the necessary and sufficient conditions for
proving that the eigenvalues and eigenfunctions of \Eq{fevp3w_hs} converge to those of \Eq{fevp3w_v} with no spurious solutions, and to estimate the order of convergence. Thus, we now proceed to establish the convergence results based on Section \ref{subsubsec:stabconv3}, as before following the notation and the ingredients of \cite{boffi2010}.
From the well posedness of problems \Eq{sourcev3}  and \Eq{sourcestab3}, for any $\f \in {\cal X}$,
we can define the operators $Z, Z_{h}: {\cal X} \rightarrow {\cal X}$ such that
 $Z \f = \u$ and $Z_h \f = \u_h$  are the displacement components of the solutions to  \Eq{sourcev3} and \Eq{sourcestab3}, respectively. 
In this way, Theorem \ref{th:thc_45} allows us to state the convergence 
\begin{align}
\Vert Z-Z_h \Vert_{{\cal L}({\cal X})} \rightarrow 0 \quad \hbox{as }~ h \rightarrow 0,\label{eq:nomT3}
\end{align} 
which is equivalent to the convergence of eigenvalues and eigenfunctions we are seeking. 


The following theorem provides the rate of convergence of the eigenvalues:

\begin{theorem}
\label{th:conv3_ev}
Assume that the continuous problem satisfies the  regularity condition
\begin{align}
\label{eq:assumpt_3genk}
 \sqrt{\mu} \Vert \u \Vert_{k''+1} + \frac{1}{\sqrt{\mu}}\Vert p  \Vert_{k''} 
 + \frac{1}{\sqrt{\mu}}\Vert \st  \Vert_{k''} \lesssim \frac{1}{\sqrt{\mu}} h^{k''}   \Vert \f \Vert_{k''-1},
\end{align}
for $k''>0$. Then have the following optimal double order of convergence  
\begin{align}
\left| \lambda-\lambda_h \right| \lesssim  \frac{\mu}{\ell^2}\left(\frac{h}{\ell}\right)^{2k''}.
\end{align}
\end{theorem}
 \begin{proof} Let ${\cal X}=(H^1_0(\Omega))^d$. From  Theorem \ref{th:thc_45} we have
\begin{align*}
\sqrt{\mu} \Vert Z \f-Z_h \f \Vert_{1}   
&=\sqrt{\mu} \Vert \u - \u_h \Vert_{1} \\
&\lesssim \varepsilon_{\text{II}}(h) \\
&\lesssim \sqrt{\mu} h^{k''}  \Vert \u\Vert_{k''+1} +\frac{1}{\sqrt{\mu}}h^{k''} \Vert \st\Vert_{k''} +\frac{1}{\sqrt{\mu}}h^{k''} \Vert p\Vert_{k''} \\
&\lesssim h^{k''}   \frac{1}{\sqrt{\mu}}\Vert \f\Vert_{k''-1} \\
&\lesssim  \frac{\ell^{2}}{\sqrt{\mu}}   \left(\frac{h}{\ell}\right)^{k''}\Vert \f\Vert_{1}.
\end{align*} 
The proof is completed by following Corollary 9.8 of \cite{boffi2010}, and observing the definitions of $\varepsilon_{\text{II}}(h)$ and $k''$. 
\end{proof}

\begin{remark}
The convergence for the choice of ${\cal X}=(L^2(\Omega))^d$  can similarly be obtained as a result of  the $L^2$-error estimate of the displacement given in Theorem~\ref{th:th3fL2} by assuming that the elliptic regularity condition holds with $k'' = 1$.
\end{remark}

Next, we make use of Corollary 9.4 of \cite{boffi2010} to conclude that if $\lambda$ is an eigenvalue of \Eq{fevp3w_v} with algebraic multiplicity $m$, $E=E(\lambda^{-1}){\cal X}$ is its generalized eigenspace, where $E(\lambda)$ is the Riesz spectral projection associated with $\lambda$, and if  $E_h=E_h(\lambda^{-1}){\cal X}$, then 
\begin{align*}
\hat{\delta}(E,E_h) \lesssim  \sup_{\substack {{\scriptsize\u}\in E \\ \Vert {\scriptsize\u}  \Vert_{{\cal X}}=1 }}\inf_{{\scriptsize\u}_h\in E_h} \Vert \u -\u_h  \Vert_{{\cal X}}.
\end{align*}
From these results we can conclude exactly the same as in Theorem~\ref{th:conv2_ef}:
\begin{theorem}
\label{th:conv3_ef}
Let $\u$ be a unit eigenfunction solution of \Eq{fevp3w_v} associated to the eigenvalue $\lambda$ of multiplicity $m$, and let ${\bm\phi}^1_h, \ldots ,{\bm \phi}^m_h$ be the eigenfunctions associated with the $m$ discrete eigenvalues solution of \Eq{fevp3w_hs} converging to $\lambda$. Then there exists a discrete eigenfunction $\u_h \in \hbox{\rm span}\{{\bm\phi}^1_h, \ldots ,{\bm\phi}^m_h\}$ such that 
\begin{align}
\Vert \u - \u_h \Vert_{1} \lesssim  h^{k''}  \Vert \u  \Vert_{k''+1}. 
\end{align}
\end{theorem}

\section{Numerical results}
\label{sec:numres}

In this section we present some numerical tests to illustrate the theoretical convergence results obtained for the two-field and three-field Stokes problems in two dimensions. Three different problem domains, namely, a square domain, an L-shaped domain, and a square with a crack, are considered in Sections \ref{subsec:square}, \ref{subsec:lshape}, and \ref{subsec:cracked}, respectively. The orthogonal subscale stabilization method is applied with equal order of $P_1$ (linear) and $P_2$ (quadratic) interpolations for all the unknowns on triangular elements.

It is important to note that all the theory about stabilized finite element methods applies for some fixed values of the constants defined in these parameters. The accuracy of the approximation for a fixed mesh size depends on the discretization type of the region as well as on the choice of the algebraic constants in the stabilization parameters. 

In the present study, the  method given in \Eq{fevp2w_hs} is applied using fixed values of the constants, that we have chosen as $c_1=1/4$ and $c_2=1/10$ for both $P_1$ and $P_2$ elements to solve the two-field eigenproblem \Eq{fevp2}. For the three-field Stokes eigenvalue problem \Eq{fevp3}, we employ the method given in  \Eq{fevp3w_hs}, where the constants of the stabilization parameters are now taken as $c_3=1$, $c_4=1/10$ and $c_5=1/4$. 

The case $\mu=1$ is considered for all the tests we examine, and as the exact solutions to the considered eigenproblems are unknown, reference values are taken from the works published for validation purposes. The reference values are given individually for each test case. We examine the convergence rates for the reference eigenvalue approximations in terms of the difference between the approximate value and the reference value, normalized by the latter. For each test case, we illustrate the results on a log-log scaled plane. 

In the simulations, the displacement (or velocity) components are taken as zero on the whole boundary, whereas the pressure is specified to be zero at a single point of the computational domain. The computations are carried out by a code written by us using  MATLAB, where the generalized eigenvalue function \textsl{eigs}, which uses ARPACK, is involved. The number of divisions in each direction is denoted by $N$. For the L-shaped domain, $N$ is the number of division in one of the shortest edges. 

\subsection{Test 1: Square domain}
\label{subsec:square}

In this test, we consider a widely used experiment, and solve the eigenproblems on the square $\Omega= [0,1] \times [0,1]$. A sample discretization of the problem domain using $N=5$ is illustrated in Figure \ref{fig:mesh1}. 
\begin{figure}[!h]\setlength{\unitlength}{1cm}
\centering
\includegraphics[width=5.5cm, height=5.3cm]{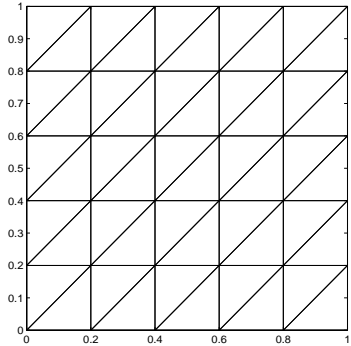}%
\caption{\small  A sample triangulation of the square domain ($N=5$). }\label{fig:mesh1}
\end{figure}
As we have already mentioned, the exact solution is unknown, and we take ${\lambda}_1 = 52.3447$ as a reference to the minimum eigenvalue (see \cite{armentano2014,liu2013,huang2011several}). 

\begin{figure}[!h]\setlength{\unitlength}{1cm}
\centering
\includegraphics[width=7cm]{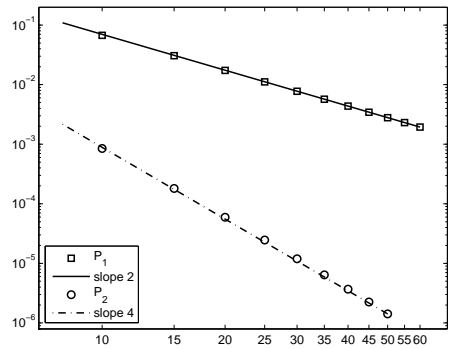}%
\caption{\small  Plot of $\log|(\lambda_1 -\lambda_h)/ \lambda_1|$ with respect to $\log|N|$for the two-field Stokes problem on a square domain.}\label{fig:f2err_sq}
\end{figure}
\begin{figure}[!h]\setlength{\unitlength}{1cm}
\centering
\includegraphics[width=7cm]{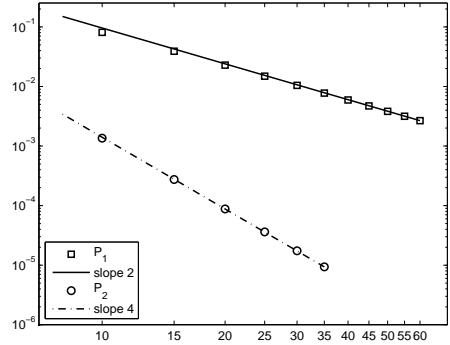}%
\caption{\small  Plot of  $\log | (\lambda_1 -\lambda_h)/ \lambda_1| $ with respect to $\log|N|$ for the three-field Stokes problem on a square domain.}\label{fig:f3err_sq}
\end{figure}

Figures \ref{fig:f2err_sq}  and \ref{fig:f3err_sq} present the convergence of the minimum eigenvalue approximations to the reference value ${\lambda}_1$  for the two-field and three-field problems, respectively. From the two figures, we can observe the optimal convergence rates, which are $2$ for $P_1$ elements and $4$ for $P_2$ elements for both approximations. These calculations prove numerically that the theoretical convergence results are achieved. To have a closer glance at the computed eigenvalues, the approximation to the first eigenvalue as well as the error values are listed in Tables \ref{tab:2fp1p1}-\ref{tab:3fp1p1} (using $P_1$ elements) and Tables \ref{tab:2fp2p2}-\ref{tab:3fp2p2} (using $P_2$ elements).  It can be seen from these tables that both for the $P_1$ and $P_2$ solutions, as the number of divisions ($N$) increases, and thus as $h$ tends to zero accordingly, the computed eigenvalues converge to the reference value. Moreover, the results show a monotonic convergence of the approximations from above.
\begin{table}[!h]
\caption{Computed eigenvalues of two-field Stokes problem on a square domain 
  using $P_1$ elements.} \label{tab:2fp1p1}
\begin{center} 
\begin{tabular}{ccc}
    \hline \hline  $N$  & $ {\lambda}_h$  & $ {(\lambda_1 -\lambda_h)/ \lambda_1}$
 \tabularnewline  \hline  
10 &  55.8688
 & 0.0673
\tabularnewline 15 & 53.9453
 &   0.0306
\tabularnewline  20 &  53.2514
 & 0.0173
\tabularnewline 25 & 52.9270
 &  0.0111
\tabularnewline  30 &  52.7498
 &    0.0077
\tabularnewline  35& 52.6426
 &    0.0057
\tabularnewline  40& 52.5729
 &    0.0044
\tabularnewline  45& 52.5251
 &   0.0034
\tabularnewline  50& 52.4908
 &    0.0028
\tabularnewline  55& 52.4655
 &    0.0023
\tabularnewline  60& 52.4462
 &    0.0019
\tabularnewline \hline
\end{tabular}
\end{center}
\end{table}
\begin{table}[!h]
\caption{Computed eigenvalues of two-field Stokes problem   on square domain
  using $P_2$ elements.} \label{tab:2fp2p2}
\begin{center} 
\begin{tabular}{ccc}
    \hline \hline  $N$  & $ {\lambda}_h$  & $ {(\lambda_1 -\lambda_h)/ \lambda_1}$
 \tabularnewline  \hline  
10 & 52.389177613831528
 & $8.4971\times 10^{-4}$
\tabularnewline 15 & 52.354184532067258
 &  $1.8119\times 10^{-4}$
\tabularnewline  20 & 52.347805305859254 
 &  $5.9324 \times 10^{-5}$
\tabularnewline 25 &  52.345990378868223
 &   $2.4652\times 10^{-5}$
\tabularnewline  30 & 52.345324052984957
 &   $1.1922 \times 10^{-5}$
\tabularnewline  35& 52.345034782505891
 &   $ 6.3957 \times 10^{-6}$
\tabularnewline  40&  52.344893303689837
 &    $ 3.6929 \times 10^{-6}$
\tabularnewline  45&  52.344817643340264
 & $ 2.2475\times 10^{-6}$
\tabularnewline  50&  52.344774270297329
 & $1.4189 \times 10^{-6} $
\tabularnewline \hline
\end{tabular}
\end{center}
\end{table}
\begin{table}[!h]
\caption{Computed eigenvalues of three-field Stokes problem   on square domain
  using $P_1$ elements.} \label{tab:3fp1p1}
\begin{center} 
\begin{tabular}{ccc}
    \hline \hline  $N$  & $ {\lambda}_h$  &  $ {(\lambda_1 -\lambda_h)/ \lambda_1}$
 \tabularnewline  \hline  
10 & 56.5919
 & 0.0811
\tabularnewline 15 & 54.3902
 &  0.0391
\tabularnewline  20 & 53.5378
 &  0.0228
\tabularnewline 25 & 53.1231 
 &  0.0149
\tabularnewline  30 & 52.8913
 &  0.0104
\tabularnewline  35& 52.7491
 &  0.0077
\tabularnewline  40& 52.6558
 &  0.0059
\tabularnewline  45& 52.5913
 &  0.0047
\tabularnewline  50& 52.5449
 &  0.0038
\tabularnewline  55& 52.5104
 &  0.0032
\tabularnewline  60& 52.4841
 &  0.0027
\tabularnewline \hline
\end{tabular}
\end{center}
\end{table}
\begin{table}[!h]
\caption{Computed eigenvalues of three-field Stokes problem  on square domain
  using $P_2$ elements.} \label{tab:3fp2p2}
\begin{center} 
\begin{tabular}{ccc}
    \hline \hline  $N$  & $ {\lambda}_h$  &  $ {(\lambda_1 -\lambda_h)/ \lambda_1}$
 \tabularnewline  \hline  
10 & 52.415573819924084
 & $1.3540 \times 10^{-3}$
\tabularnewline 15 & 52.359070017800590
 &  $2.7453 \times 10^{-4}$ 
\tabularnewline  20 & 52.349305192050018
 &  $8.7978 \times 10^{-5}$
\tabularnewline 25 & 52.346595128313346
 &   $3.6205\times 10^{-5}$
\tabularnewline  30 & 52.345613136524591
 &   $1.7445\times 10^{-5}$
\tabularnewline  35& 52.345190028331487
 &   $9.3616\times 10^{-6}$
\tabularnewline \hline
\end{tabular}
\end{center}
\end{table}

\begin{table}[!h]
\caption{Computed ten eigenvalues of two-field Stokes problem on a square domain
  using $P_1$ elements.} \label{tab:eg10_2fp1}
\begin{center} 
\scriptsize
\begin{tabular}{cccccccc}
\hline \hline Ref. & $N=10$ & $N=15$ & $N=20$ & $N=25$ & $N=30$ & $N=35$ & $N=40$\tabularnewline\hline 
  52.3447 & 55.8688 & 53.9453 & 53.2514 & 52.927 & 52.7498 & 52.6426 & 52.5729\tabularnewline
\hline 
   92.1245 & 99.9955 & 95.7656 & 94.1952 & 93.4560 & 93.0514 & 92.8065 & 92.6471\tabularnewline
\hline 
   92.1246 & 104.6259 & 97.7599 & 95.3019 & 94.1591 & 93.5375 & 93.1626 & 92.9192\tabularnewline
\hline 
  128.2100 & 148.7460 & 138.2263 & 133.9922 & 131.9494 & 130.8203 & 130.1333 & 129.6851\tabularnewline
\hline 
  154.1260 & 179.5074 & 165.7321 & 160.7009 & 158.3444 & 157.0584 & 156.2813 & 155.7763\tabularnewline
\hline 
  167.0298 & 196.3993 & 179.8558 & 174.1717 & 171.5767 & 170.1783 & 169.3389 & 168.7957\tabularnewline
\hline 
  189.5729 & 221.5153 & 205.7600 & 199.0557 & 195.7457 & 193.8967 & 192.7654 & 192.0246\tabularnewline
\hline 
  189.5735 & 240.9553 & 214.2593 & 203.7305 & 198.6940 & 195.9248 & 194.2460 & 193.1532\tabularnewline
\hline 
  246.3240 & 303.4553 & 271.6308 & 260.4574 & 255.3276 & 252.5584 & 250.8957 & 249.8195\tabularnewline
\hline 
  246.3243 & 304.9802 & 275.3703 & 262.4826 & 256.6058 & 253.4404 & 251.5414 & 250.3128\tabularnewline
\hline 
\end{tabular}
\end{center}
\end{table}

\begin{table}[!h]
\caption{Computed ten eigenvalues of two-field Stokes problem on a square domain
  using $P_2$ elements.} \label{tab:eg10_2fp2}
\begin{center} 
\scriptsize
\begin{tabular}{cccccccc}
\hline \hline Ref. & $N=10$ & $N=15$ & $N=20$ & $N=25$ & $N=30$ & $N=35$ & $N=40$\tabularnewline\hline 
  52.3447 & 52.3892 & 52.3542 & 52.3478 & 52.3460 & 52.3453 & 52.3450 & 52.3449\tabularnewline
\hline 
   92.1245 & 92.2650 & 92.1540 & 92.1341 & 92.1285 & 92.1264 & 92.1255 & 92.1250\tabularnewline
\hline 
   92.1246 & 92.3546 & 92.1731 & 92.1402 & 92.1310 & 92.1276 & 92.1261 & 92.1254\tabularnewline
\hline 
  128.2100 & 128.8179 & 128.3406 & 128.2526 & 128.2276 & 128.2184 & 128.2144 & 128.2124\tabularnewline
\hline 
  154.1260 & 154.7857 & 154.2660 & 154.1712 & 154.1445 & 154.1347 & 154.1305 & 154.1284\tabularnewline
\hline 
  167.0298 & 167.8012 & 167.1932 & 167.0829 & 167.0516 & 167.0401 & 167.0351 & 167.0327\tabularnewline
\hline 
  189.5729 & 190.8794 & 189.8582 & 189.6665 & 189.6116 & 189.5913 & 189.5825 & 189.5781\tabularnewline
\hline 
  189.5735 & 191.5500 & 190.0079 & 189.7160 & 189.6322 & 189.6013 & 189.5879 & 189.5813\tabularnewline
\hline 
  246.3240 & 248.3017 & 246.7483 & 246.4620 & 246.3806 & 246.3507 & 246.3377 & 246.3314\tabularnewline
\hline 
  246.3243 & 248.6870 & 246.8347 & 246.4907 & 246.3926 & 246.3566 & 246.3409 & 246.3332\tabularnewline
\hline 
\end{tabular}
\end{center}
\end{table}
Furthermore, we want to look at the first ten eigenvalue approximations with comparison to the reference values obtained by the standard Galerkin method using $P_2$-$P_1$ interpolations satisfying the appropriate inf-sup condition on a fine mesh ($N=60$). 
The results are shown in Table \ref{tab:eg10_2fp1} and Table \ref{tab:eg10_2fp2} using respectively $P_1$ elements and $P_2$ elements for the two-field case. The results for the three-field case are shown in Table \ref{tab:eg10_3fp1} (using $P_1$ elements) and Table \ref{tab:eg10_3fp2} (using $P_2$ elements). The numerical results show that the approximations for all the first ten eigenvalues in the calculated spectrum converge to the corresponding reference solutions, and the approximated values are above the reference solutions for all cases. 

\begin{table}[!h]
\caption{Computed ten eigenvalues of three-field Stokes problem on a square domain
  using $P_1$ elements.} \label{tab:eg10_3fp1}
\begin{center} 
\scriptsize
\begin{tabular}{cccccccc}
\hline \hline Ref. & $N=10$ & $N=15$ & $N=20$ & $N=25$ & $N=30$ & $N=35$ & $N=40$\tabularnewline\hline 
52.3447 & 56.5919 & 54.3902 & 53.5378 & 53.1231 & 52.8913 & 52.7491 & 52.6558\tabularnewline
\hline 
   92.1245 & 98.9870 & 95.8558 & 94.4066 & 93.6468 & 93.2066 & 92.9310 & 92.7479\tabularnewline
\hline 
   92.1246 & 106.0693 & 98.9834 & 96.1631 & 94.7706 & 93.9869 & 93.5043 & 93.1867\tabularnewline
\hline 
  128.2100 & 148.9444 & 140.1753 & 135.6434 & 133.2037 & 131.7734 & 130.8721 & 130.2706\tabularnewline
\hline 
  154.1260 & 172.7031 & 164.8846 & 160.9287 & 158.7355 & 157.4291 & 156.5992 & 156.0429\tabularnewline
\hline 
  167.0298 & 189.4507 & 179.0202 & 174.4194 & 171.9815 & 170.5594 & 169.6654 & 169.0694\tabularnewline
\hline 
  189.5729 & 212.7258 & 205.2605 & 199.9178 & 196.7485 & 194.7888 & 193.5145 & 192.6472\tabularnewline
\hline 
  189.5735 & 238.4948 & 218.2487 & 207.5491 & 201.7170 & 198.2669 & 196.0809 & 194.6171\tabularnewline
\hline 
  246.3240 & 269.6729 & 263.5867 & 258.3683 & 254.8240 & 252.5404 & 251.0333 & 250.0011\tabularnewline
\hline 
  246.3243 & 282.5791 & 268.1758 & 260.8985 & 256.4580 & 253.6864 & 251.8820 & 250.6549\tabularnewline
\hline 
\end{tabular}
\end{center}
\end{table}
\begin{table}[!h]
\caption{Computed ten eigenvalues of three-field Stokes problem on a square domain
  using $P_2$ elements.} \label{tab:eg10_3fp2}
\begin{center} 
\scriptsize
\begin{tabular}{ccccccc}
\hline \hline Ref. & $N=10$ & $N=15$ & $N=20$ & $N=25$ & $N=30$ & $N=35$ \tabularnewline\hline 
 52.3447 & 52.4156 & 52.3591 & 52.3493 & 52.3466 & 52.3456 & 52.3452  \tabularnewline
\hline 
   92.1245 & 92.2684 & 92.1548 & 92.1343 & 92.1285 & 92.1264 & 92.1255 \tabularnewline
\hline 
   92.1246 & 92.4690 & 92.1927 & 92.1461 & 92.1333 & 92.1287 & 92.1267 \tabularnewline
\hline 
  128.2100 & 129.306 & 128.4260 & 128.2782 & 128.2378 & 128.2232 & 128.2170 \tabularnewline
\hline 
  154.1260 & 154.6511 & 154.2415 & 154.1636 & 154.1414 & 154.1332 & 154.1297 \tabularnewline
\hline 
  167.0298 & 167.5411 & 167.1493 & 167.0696 & 167.0462 & 167.0375 & 167.0337 \tabularnewline
\hline 
  189.5729 & 191.4011 & 189.9504 & 189.6936 & 189.6223 & 189.5964 & 189.5852 \tabularnewline
\hline 
  189.5735 & 193.3702 & 190.3109 & 189.8036 & 189.6665 & 189.6175 & 189.5965 \tabularnewline
\hline 
  246.3240 & 246.4209 & 246.4648 & 246.3805 & 246.3486 & 246.3356 & 246.3297 \tabularnewline
\hline 
  246.3243 & 246.8890 & 246.5649 & 246.4125 & 246.3617 & 246.3419 & 246.3331 \tabularnewline
\hline 
\end{tabular}
\end{center}
\end{table}
\begin{figure}[!h]\setlength{\unitlength}{0.85cm}
\centering
\includegraphics[width=9.0cm]{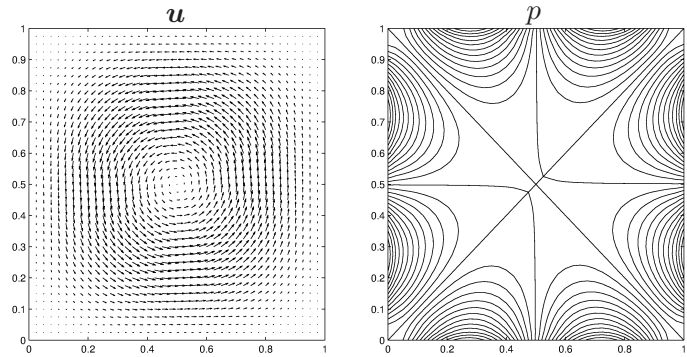}%
\begin{picture}(0,0)(0,0)
\put(-8.0,5.3){\makebox(0,0){\small $\u$}}
\put(-2.4,5.3){\makebox(0,0){\small  $p$}}
\end{picture}%
\caption{\small  Plots of $\u$ and $p$ for $N=40$ with $P_1$ elements (two-field, square domain).}
\label{fig:velp2fsqN40}
\end{figure}
\begin{figure}[!h]\setlength{\unitlength}{1cm}
\centering
\includegraphics[width=12cm]{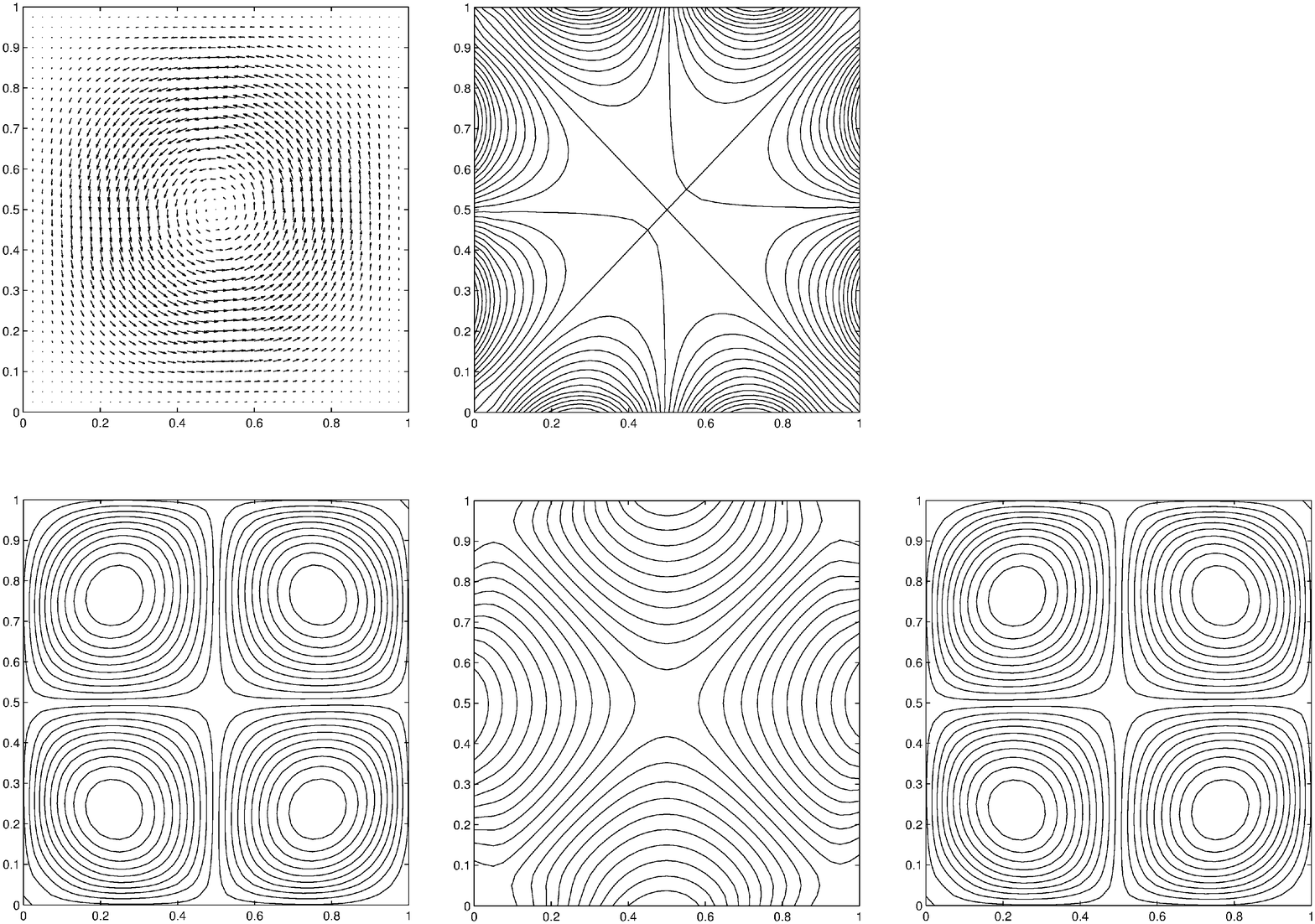}%
\begin{picture}(0,0)(0,0)
\put(-10.0,8.45){\makebox(0,0){\small $\u$}}
\put(-5.9,8.45){\makebox(0,0){\small  $p$}}
\put(-10.0,4.0){\makebox(0,0){\small $\sigma_{11}$}}
\put(-5.9,4.0){\makebox(0,0){\small $\sigma_{12}$}}
\put(-1.8,4.0){\makebox(0,0){\small $\sigma_{22}$}}
\end{picture}%
\caption{\small  Plots of $\u$, $p$ and $\st$-components for $N=40$ with $P_1$ elements (three-field, square domain).}
\label{fig:velpsigmaf3N40sq}
\end{figure}

In order to compare our results qualitatively, we plot the unknowns, when $N=60$ using $P_1$ elements in Figure \ref{fig:velp2fsqN40} for the two-field case, and in Figure \ref{fig:velpsigmaf3N40sq} for the three-field case.  
Comparing these two figures, one can see the perfect agreement in the velocity and pressure profiles obtained with the two formulations. Moreover, we can observe that the behavior of the velocity streamlines and pressure levels are in good agreement with the previously published results \cite{armentano2014,huang2011several}.

Before proceeding, we want to report an unexpected behavior we have encountered during our numerical experiments. As the numerical analysis for both two-field and three-field source problems suggests, the numerical constants in the stabilization parameters can be arbitrarily chosen in a wide range. Considering the source problems, for all cases this conclusion has been validated by testing different combinations of the parameters chosen from a very large interval. For the eigenvalue problems, this is also true for the two-field case. However, considering the approximation of first ten eigenvalues for the three-field eigenproblem, when we test the method with constants approximately ten times larger than our default values, we have observed that for certain cases spurious node-to-node oscillations in the approximations are developed. This bad behavior only exists in the seventh and tenth modes, and only for the three-field case for both $P_1$ and $P_2$ elements, using high values for the algorithmic constants. {In other words, the correct values are well approximated for larger values of the stabilization constants for the two-field problem; however, a bad behavior is observed in two approximations of the first ten eigenvalues for the three-field case. These results lead us to think that} a possible reason for this issue could be related to the deficiency of the algorithm that computes the eigenvalues for the structure of the resulting system {in the three-field case}. 

\subsection{Test 2: L-shaped domain}
\label{subsec:lshape}

In the previous example we have considered a convex domain and showed that the convergence estimates are recovered numerically for both two-field and three-field cases. Next, we want to examine a test case with an L-shaped domain with a re-entrant corner, defined by $\Omega= [-1,1]^2 \setminus [0,1]^2$. The problem domain with a discretization where $N=5$ is shown in Figure \ref{fig:mesh2}.
\begin{figure}[!h]\setlength{\unitlength}{1cm}
\centering
\includegraphics[width=5.8cm, height=5.3cm]{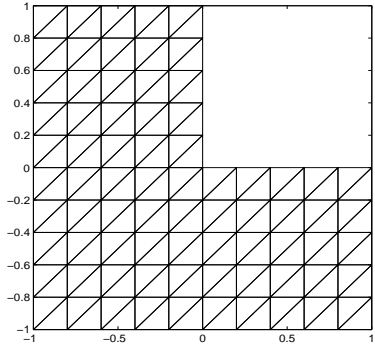}%
\caption{\small  A sample triangulation of  the L-shaped domain ($N=5$). }\label{fig:mesh2}
\end{figure}

For this experiment, we consider ${\lambda}_2 = 48.9844$ as the reference value to the fourth eigenvalue. It is known that the dual problem has $H^{\iota +1}$ regularity where $0<\iota<1$ \cite{liu2013}. 
The convergence results obtained for the two-field problem are shown in Figure \ref{fig:f2errL}, where the reference values are given in Table \ref{tab:2fp1p1_l} and \ref{tab:2fp2p2_l} using $P_1$ and $P_2$ elements, respectively. Similarly, 
Figure \ref{fig:f3errL}  plots the convergence results for the three-field case, whereas the approximated values are listed in Table \ref{tab:3fp1p1_l} ($P_1$ results) and Table \ref{tab:3fp2p2_l} ($P_2$ results).  We conclude from these results that the method achieves a double order of convergence from above, for the errors of the approximated eigenvalues. Further, we can infer that the reference eigenfunction corresponding to the fourth eigenvalue is smooth, complying with the results reported in \cite{liu2013}.  

\begin{table}[!h]
\caption{Computed eigenvalues of the two-field Stokes problem   
 on a L-shaped domain using $P_1$ elements.} \label{tab:2fp1p1_l}
\begin{center} 
\begin{tabular}{ccc}
    \hline \hline  $N$  & $  \lambda_h$  & $ {(\lambda_2 -\lambda_h})/ \lambda_2$
 \tabularnewline  \hline  
5 &  58.6756
 &   0.1978
\tabularnewline 10 & 51.8885 
 &  0.0593     
\tabularnewline  15 &   50.3119
 &  0.0271  
\tabularnewline  20 & 49.7384
 &    0.0154
\tabularnewline  25 &  49.4692
 &     0.0099 
\tabularnewline  30 &  49.3218
 &   0.0069
\tabularnewline \hline
\end{tabular}
\end{center}
\end{table}
 
\begin{table}[!h]
\caption{Computed eigenvalues of the two-field Stokes problem   
on a L-shaped domain  using $P_2$ elements.} \label{tab:2fp2p2_l}
\begin{center} 
\begin{tabular}{ccc}
     \hline \hline  $N$  & $  \lambda_h$  & $ {(\lambda_2 -\lambda_h})/ \lambda_2$
 \tabularnewline  \hline  
5 &  49.8045
 &  0.0167
\tabularnewline 10 &  49.0428
 & 0.0012      
\tabularnewline  15 &  48.9959
 & 0.0002  
\tabularnewline  20 & 48.9877
 & 0.0001
 
\tabularnewline \hline
\end{tabular}
\end{center}
\end{table}

\begin{table}[!h]
\caption{Computed eigenvalues of the three-field Stokes problem   
 on a L-shaped domain  using $P_1$ elements.} \label{tab:3fp1p1_l}
\begin{center} 
\begin{tabular}{ccc}
    \hline \hline  $N$  & $  \lambda_h$  & $ {(\lambda_2 -\lambda_h})/ \lambda_2$
 \tabularnewline  \hline  
5 & 51.9184  
 & 0.0599   
\tabularnewline 10 & 49.8498 
 & 0.0177      
\tabularnewline  15 & 49.4469   
 & 0.0094   
\tabularnewline  20 & 49.2607 
 & 0.0056   
\tabularnewline  25 & 49.1658 
 & 0.0037    
\tabularnewline  30 & 49.1120  
 & 0.0026  
\tabularnewline \hline
\end{tabular}
\end{center}
\end{table}
 
\begin{table}[!h]
\caption{Computed eigenvalues of the three-field Stokes problem   
 on a L-shaped domain  using $P_2$ elements.} \label{tab:3fp2p2_l}
\begin{center} 
\begin{tabular}{ccc}
     \hline \hline  $N$  & $  \lambda_h$  & $ {(\lambda_2 -\lambda_h})/ \lambda_2$
 \tabularnewline  \hline  
5 &  49.4628  
 &  0.0098
\tabularnewline 10 &   49.0224
 &  0.0008    
\tabularnewline  15 &    48.9923
 &  0.0002  
\tabularnewline  20 &  48.9867
 & 0.0000
\tabularnewline \hline
\end{tabular}
\end{center}
\end{table}

\begin{figure}[!h]\setlength{\unitlength}{1cm}
\centering
\includegraphics[width=7cm]{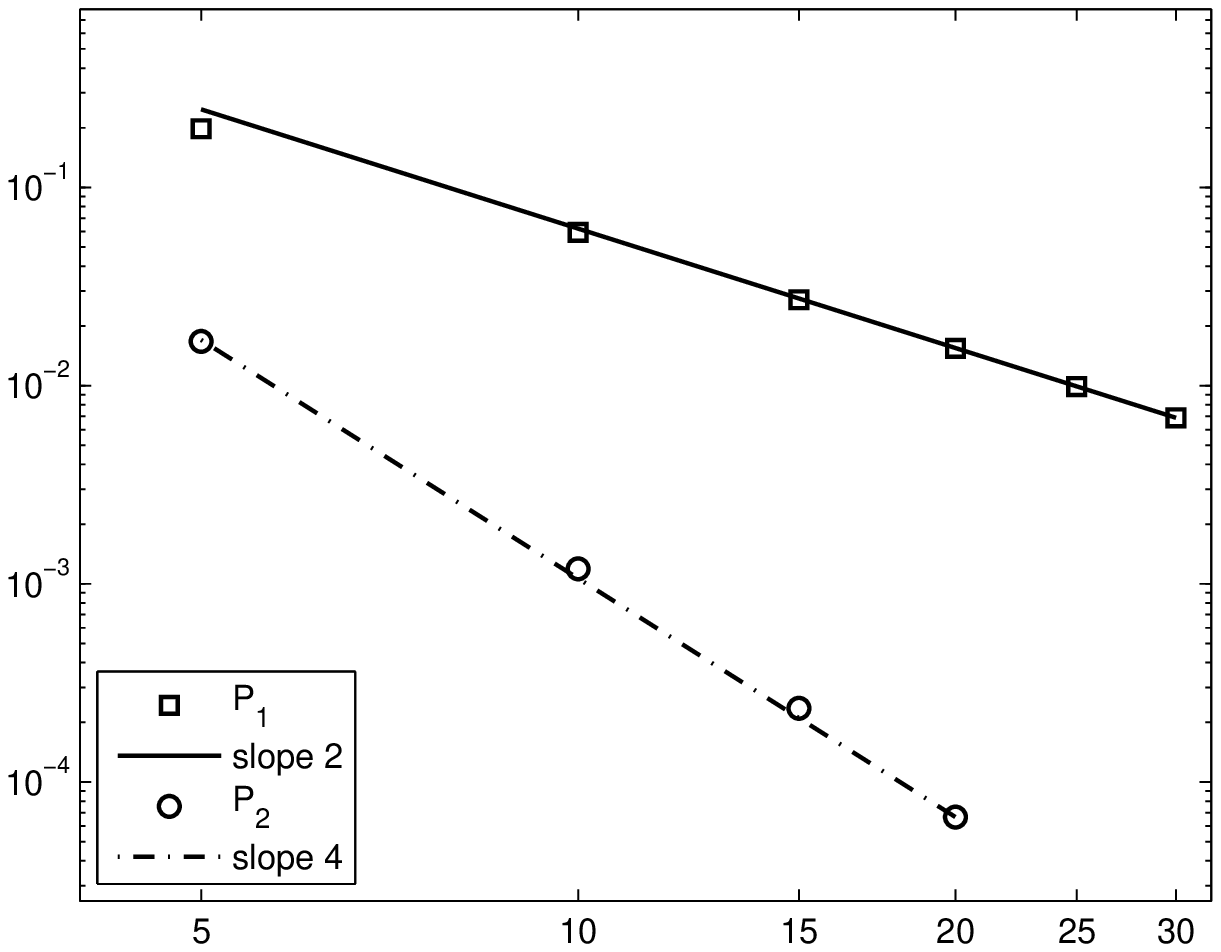}%
\caption{\small  Plot of  $\log | (\lambda_2 -\lambda_h)/ \lambda_2| $ with respect to $\log|N|$ for the two-field Stokes problem on a L-shaped domain.}\label{fig:f2errL}
\end{figure}

\begin{figure}[!h]\setlength{\unitlength}{1cm}
\centering
\includegraphics[width=7cm]{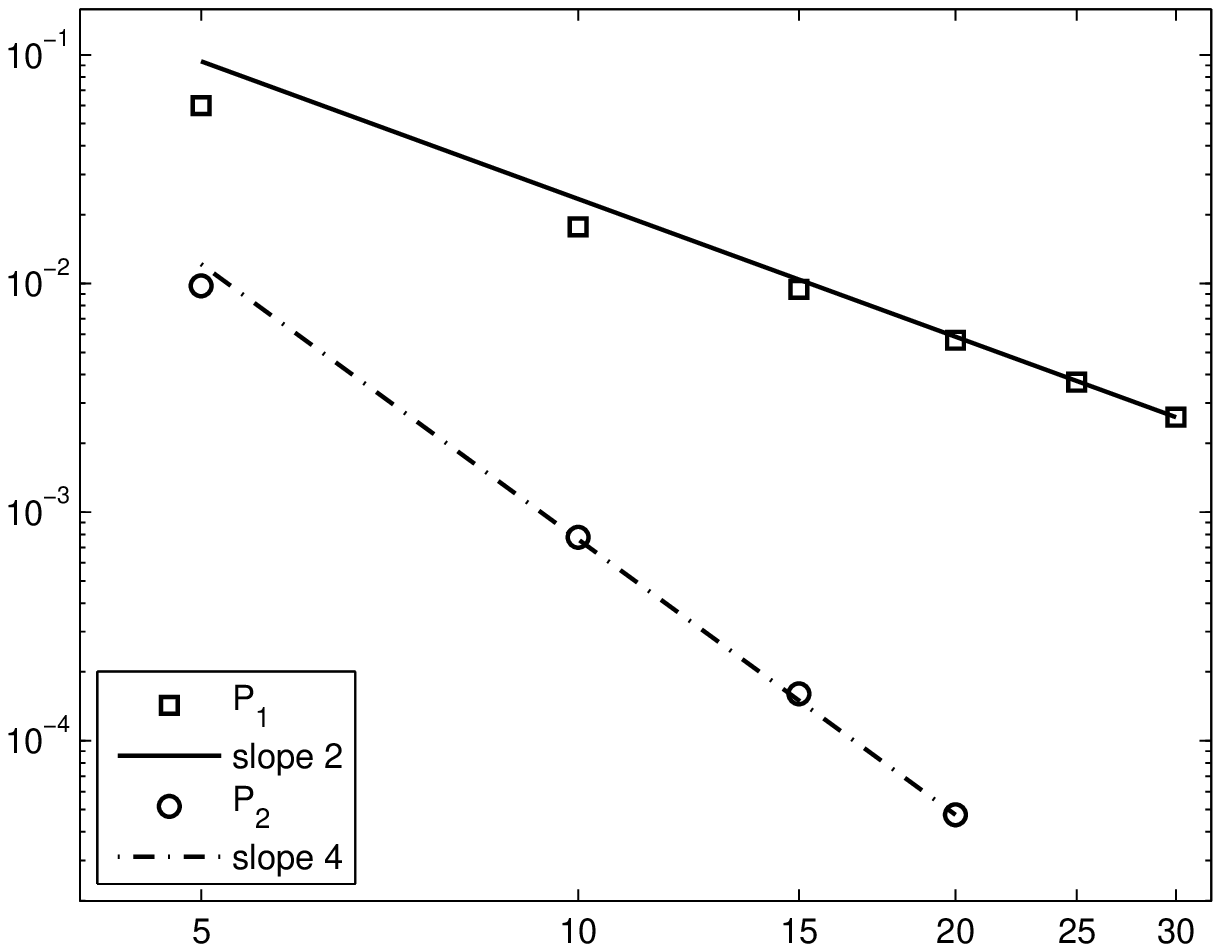}%
\caption{\small  Plot of  $\log | (\lambda_2 -\lambda_h)/ \lambda_2| $ with respect to $\log|N|$ for the three-field Stokes problem on a L-shaped domain.}\label{fig:f3errL}
\end{figure}

\subsection{Test 3: Cracked square domain}
\label{subsec:cracked}

Having dealt with two examples having analytic solutions, we consider another domain with a re-entrant corner, namely, a square with a 45-degrees crack, as the last test. The problem domain is discretized by a sequence of unstructured triangular meshes, and the total number of vertices is denoted by $M$. Figure \ref{fig:mesh3} shows the problem domain and a sample discretization where $M=136$. 

We take $\lambda_3=31.2444$ as the reference solution to the first eigenvalue for this experiment. The corresponding solution is known to be singular \cite{liu2013}. 

\begin{figure}[!h]\setlength{\unitlength}{1cm}
\centering
\includegraphics[width=5.8cm, height=5.3cm]{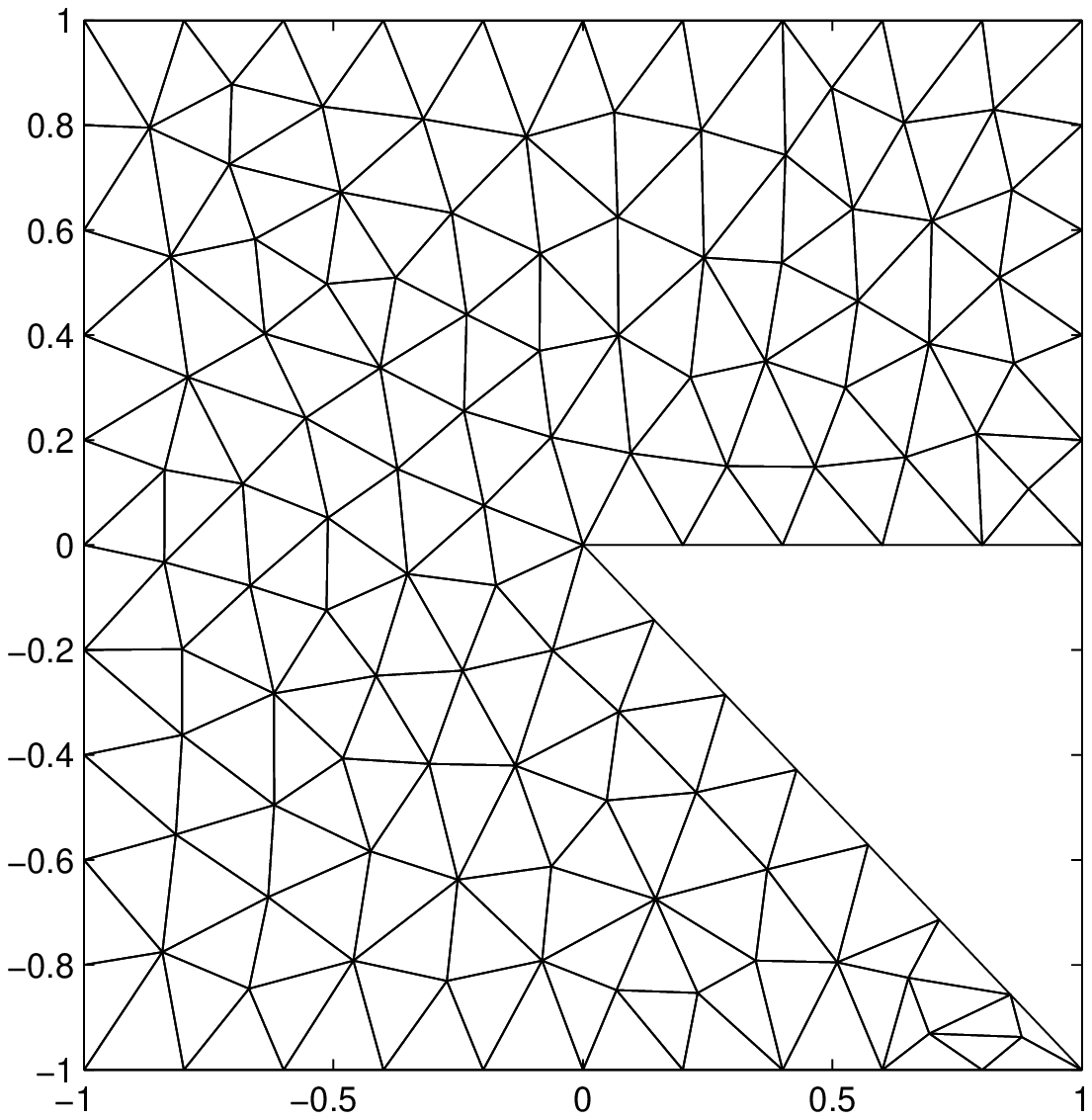}%
\caption{\small  A sample triangulation of  the cracked square ($M=136$). }\label{fig:mesh3}
\end{figure}

Tables \ref{tab:2fp1p1_cracked} and \ref{tab:2fp2p2_cracked} list the first eigenvalue approximations together with the relative errors, using $P_1$ and $P_2$ elements, respectively, for the two-field case. Similarly, in Tables \ref{tab:3fp1p1_cracked} and  \ref{tab:3fp2p2_cracked} we present the results for the three-field case.  

\begin{table}[!h]
\caption{Computed eigenvalues of the two-field Stokes problem   
 on a cracked square using $P_1$ elements.} \label{tab:2fp1p1_cracked}
\begin{center} 
\begin{tabular}{ccc}
    \hline \hline  $M$  & ${\lambda}_h$  &  $ {(\lambda_3 -\lambda_h)/ \lambda_3}$
 \tabularnewline  \hline  
136 & 33.9006
 & 0.0850
\tabularnewline 477 &  32.1248
 &  0.0282     
\tabularnewline  989  &   31.6912
 & 0.0143  
\tabularnewline  1861 &   31.4876 
 &  0.0078 
\tabularnewline  2515 &  31.4485
 &   0.0065 
\tabularnewline  3489 &  31.4013
 &    0.0050
\tabularnewline \hline
\end{tabular}
\end{center}
\end{table}
\begin{table}[!h]
\caption{Computed eigenvalues of the two-field Stokes problem   
 on a cracked square using $P_2$ elements.} \label{tab:2fp2p2_cracked}
\begin{center} 
\begin{tabular}{ccc}
    \hline \hline  $M$  & $ {\lambda}_h$  & $ {(\lambda_3 -\lambda_h)/ \lambda_3}$
 \tabularnewline  \hline  
136 &  31.4697
 &  0.0072  
\tabularnewline 477 &  31.3694  
 &   0.0040 
\tabularnewline  989 & 31.3215 
 &    0.0025 
\tabularnewline  1861 & 31.3102 
 & 0.0021   
\tabularnewline  2515 & 31.3074
 &  0.0020  
\tabularnewline  3489 & 31.3022 
 &   0.0018 

\tabularnewline \hline
\end{tabular}
\end{center}
\end{table}
\begin{table}[!h]
\caption{Computed eigenvalues of the three-field Stokes problem   
 on a cracked square using $P_1$ elements.} \label{tab:3fp1p1_cracked}
\begin{center} 
\begin{tabular}{ccc}
    \hline \hline  $M$  & $ {\lambda}_h$  & $ {(\lambda_3 -\lambda_h)/ \lambda_3}$
 \tabularnewline  \hline  
136 & 35.5336  
 &  0.1373 
\tabularnewline 477&   33.6856 
 & 0.0781      
\tabularnewline  989   &  32.6086  
 & 0.0437
\tabularnewline  1861 &   32.0648    
 & 0.0263 
\tabularnewline  2515 & 31.9739   
 &   0.0233
\tabularnewline  3489 &  31.8046   
 &  0.0179
\tabularnewline \hline
\end{tabular}
\end{center}
\end{table}
\begin{table}[!h]
\caption{Computed eigenvalues of the three-field Stokes problem   
 on a cracked square using $P_2$ elements.} \label{tab:3fp2p2_cracked}
\begin{center} 
\begin{tabular}{ccc}
    \hline \hline  $M$  & $ {\lambda}_h$  & $ {(\lambda_3 -\lambda_h)/ \lambda_3}$
 \tabularnewline  \hline  
136 &  32.0113   
 &  0.0245
\tabularnewline 477 &  31.6124 
 &  0.0118    
\tabularnewline  989  &  31.4632
 & 0.0070 
\tabularnewline 1861 &  31.4086  
 &  0.0053   
\tabularnewline \hline
\end{tabular}
\end{center}
\end{table}

The convergence results in terms of the errors are displayed in Figure \ref{fig:f2err_cr} for the two-field case and in Figure \ref{fig:f3err_cr} for the three-field case. The results show that the monotonic approximation property of the method is also preserved for this example, and the approximation orders are higher than the reference value for all cases considered. The figures indicate that the asymptotic regime has not been reached yet, and the number of elements has to be further increased  in order to  obtain a linear dependence of the error on the number of total vertices. We clearly infer that the convergence order has decreased for the problem where we do not have global regularity, and the solution is not analytic. Thus, we can conclude that the convergence is driven by the regularity of the continuous solution, as expected.

\begin{figure}[!h]\setlength{\unitlength}{1cm}
\centering
\includegraphics[width=7cm]{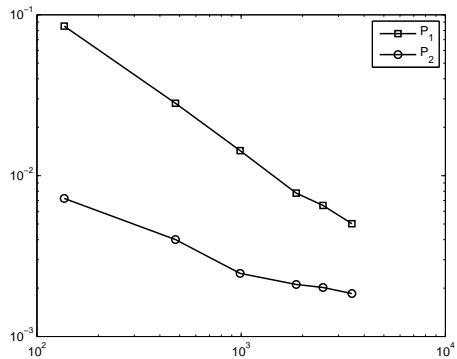}%
\caption{\small  Plot of $\log|(\lambda_3 -\lambda_h)/ \lambda_3|$ with respect to $\log|M|$for the two-field Stokes problem on a cracked square domain.}\label{fig:f2err_cr}
\end{figure}
\begin{figure}[!h]\setlength{\unitlength}{1cm}
\centering
\includegraphics[width=7cm]{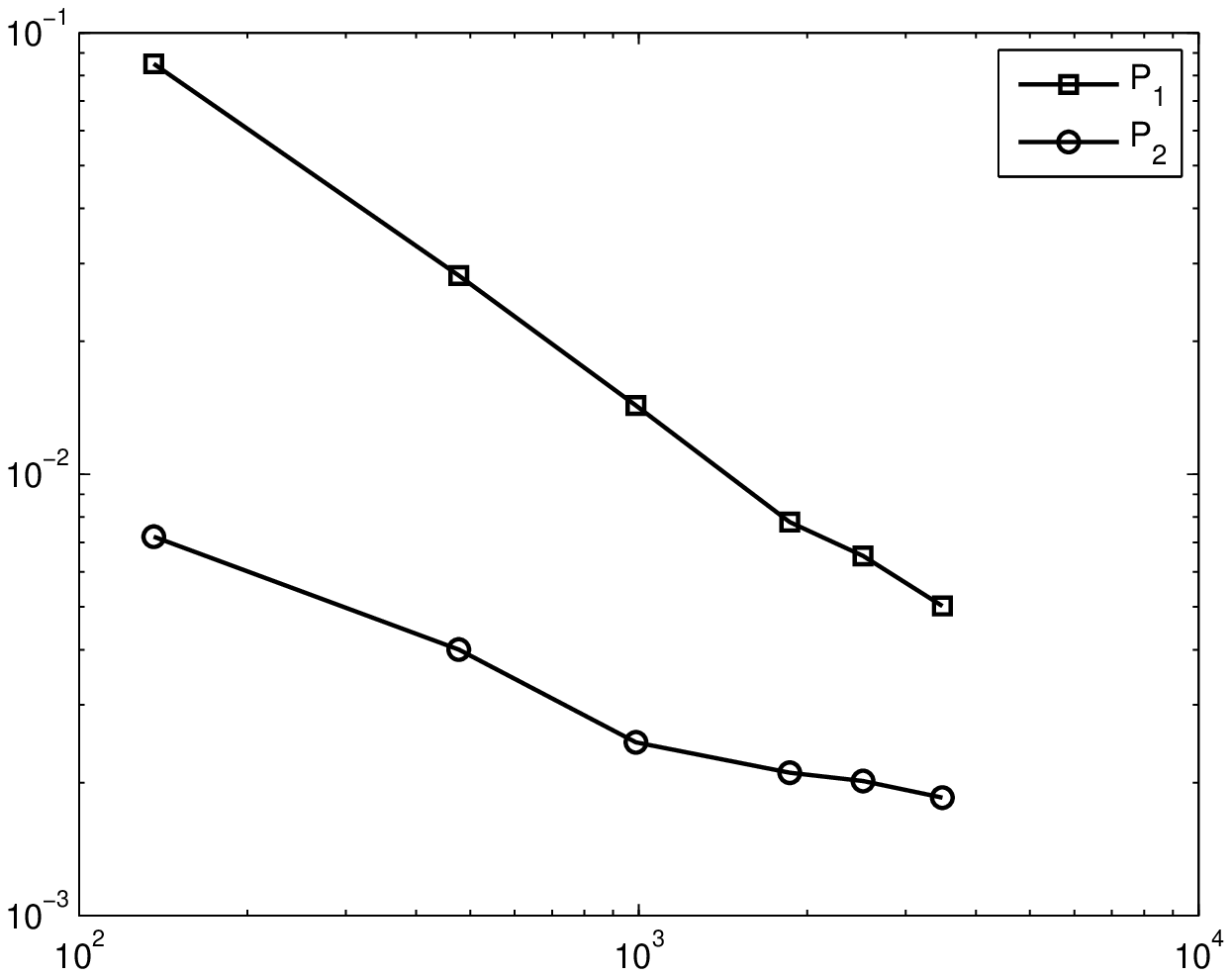}%
\caption{\small  Plot of  $\log | (\lambda_3 -\lambda_h)/ \lambda_3| $ with respect to $\log|M|$ for the three-field Stokes problem on a cracked square domain.}\label{fig:f3err_cr}
\end{figure}

\section{Conclusions}
\label{sec:conc}

The stabilized finite element formulation based on the  application of subgrid scale concept to the two-field and three-field Stokes eigenvalue problems has been presented. The virtue of the method relies in considering the subscales orthogonal to the finite element space; the fact that the orthogonal projection of the displacements (or velocities) vanishes provides an essential property which makes the method very convenient for eigenvalue problems. The finite element approximation to the three-field Stokes eigenvalue problem is  another novel contribution of the paper. The convergence and error estimates are based on the finite element analysis of the corresponding source problems.  The formulations are shown to be optimally convergent for a given set of algorithmic parameters on which the methods depend. The numerical computations show that the accuracy of the method is the one expected from the convergence analysis, and the theoretical convergence rates for all the experiments considered are exactly achieved in the numerical results presented.

\section*{Acknowledgments}

The first author acknowledges the support received from the Scientific and Technological Research Council of Turkey (TUBITAK) 2219-Postdoctoral Research Program Grant. The second author was partially supported by PRIN/MIUR, by GNCS/INDAM,
and by IMATI/CNR. The third author is grateful to the ICREA Acad\`emia Program, from the Catalan government.

\bibliographystyle{ieeetr}
\bibliography{latexbi}

\end{document}